\pdfoutput=1
     \documentclass{article}
\usepackage{mystuff_new}  

\usepackage{algorithm,algorithmic}
\newcommand{\Nt}{M}
\newcommand{\Np}{N}
\newcommand{\tstep}{h}
\newcommand{\tvecf}{\skew{4}\tilde{\vec f}}
\newcommand{\tG}{\widetilde{G}}

\newcommand{\Cost}[1]{\mathrm{Cost}^{(\operatorname{#1})}}
\newcommand{\expect}[1]{\mean{#1}}
\newcommand{\variance}[1]{\operatorname{Var}\bp{#1}}
\newcommand{\Ncoarse}{\Nt_{0}}
\newcommand{\MLMC}{\mathrm{MLMC}}
\newcommand{\MC}{\mathrm{MC}}
\newcommand{\Yhat}{\widehat{Y}}

\newcommand{\Vhat}{\widehat{V}}
\newcommand{\QOI}{\mathcal{P}}
\newcommand{\Phat}{\widehat{\QOI}}
\newcommand{\Ptilde}{\widetilde{\QOI}}

\newcommand\Cpinfty{C^\infty_{\mathrm{poly}}}
\newcommand\rmF{{\mathrm{F}}}
\newcommand\erm{{\mathrm{e}}}
\newcommand\kB{{k_\mathrm{B}}}

\title{Improving MLMC for SDEs with application to the
  Langevin equation} \author{Eike H. M\"uller \and Rob
  Scheichl \and Tony Shardlow}

\begin{document}
\maketitle
\begin{abstract}
  { This paper applies several well-known tricks from the
    numerical treatment of deterministic differential
    equations to improve the efficiency of the Multilevel
    Monte Carlo (MLMC) method for stochastic differential
    equations (SDEs) and especially the Langevin
    equation. We use modified equations analysis to
    circumvent the need for a strong-approximation theory
    for the integrator, and we apply this to introduce MLMC
    for Langevin-type equations with integrators based on
    operator splitting. We combine this with extrapolation
    and 
    investigate the use of discrete random variables in
    place of the Gaussian increments, which is a well-known
    technique for the weak approximation of SDEs. We show
    that, for small-noise problems, discrete random variables
    can lead to an increase in efficiency of almost two
    orders of magnitude for practical levels of accuracy.}
  \begin{description}
  \item[Keywords] numerical solution of stochastic
    differential equations, modified equations, 
    geometric integrators, weak approximation, extrapolation.
  \end{description}
\end{abstract}

\section{Introduction}

This paper is concerned with the numerical solution of
stochastic differential equations (SDEs) by the Multilevel
Monte Carlo (MLMC) method.  MLMC
\citep{Heinrich2001,MR2436856} is an important
variance-reduction method that is well established by now
and has been successfully applied to a wide class of
problems in stochastic simulation and in uncertainty
quantification; for example, \citep{GilesSzpruch2013,
  Cliffe2011,Dereich2011,Barthetal2011,GilesReisinger2012,Mishraetal2012,
  Hoeletal2012,Andersonetal2011}. The variance reduction in
MLMC is achieved by computing approximations of the solution
on different ``levels'' consisting, in the SDE case, of
numerical integrators with different time-step sizes. These
computations are then combined in an efficient way to define
a multilevel estimator for the moments that has a smaller
variance than the standard Monte Carlo estimator and can
therefore be computed faster.  

Let $(\Omega, {\cal F}, \sprob)$ denote a probability space
and let $\smean$ and $\operatorname{Var}$ denote the
expectation and variance with respect to $\sprob$.  Consider
first the initial-value problem
\begin{equation}\label{eq:sde}
d\vec X%
=\vec f(\vec X)\,dt+ G(\vec X)\,d\vec W(t),\qquad%
\vec X(0)%
=\vec X_0,
\end{equation}
for $\vec f\colon \real^d \to \real^d$ and $G\colon \real^d
\to \real^{d\times m}$ and initial data $\vec
X_0\in\real^d$. Here $\vec W(t)$ is a vector of $m$ \iid
standard Brownian motions on $(\Omega,{\cal F},\sprob)$.
Suppose that there exists a well-defined solution $\vec
X(t)$ when \cref{eq:sde} is interpreted as an Ito integral
equation. For simplicity, we only consider approximating
moments of the solution at a prescribed end time as the
quantities of interest, but other, more complicated
functionals could also be studied. That is, we are
interested in computing $\mean{\phi(\vec X(T))}$ for some
$\phi\colon \real^d\to\real$ and time $T>0$. Consider the
approximation by a sequence of random variables $\vec
X_n\approx \vec X(t_n)$ for $t_n=n\tstep$ with
$n\in\naturals$ and a time step $\tstep$.  For example,
$\vec X_n$ may result from the Euler--Maruyama method
\begin{equation}
  \label{eq:1}
  \vec X_{n+1}%
  =\vec X_n+ \vec f(\vec X_n)\tstep%
  + G(\vec  X_n)\sqrt{\tstep}\, \vec\xi_n,  
\end{equation}
with $\vec\xi_n\sim \Nrm(0,I)$ \iid. In this case, $\vec X_n$ is 
a weak first-order approximation to $\vec X(t_n)$ so that, for any
$\phi\colon\real^d\to\real$ in a suitable class of test functions ${\cal C}$,
\[
\sup_{0\le t_n\le T}%
\mean{\phi(\vec X(t_n))}-\mean{\phi(\vec X_n)}%
=\order{\tstep}.
\]
If $\vec f$ and $G$ are sufficiently smooth, $\mathcal{C}$
contains all infinitely differentiable functions whose
derivatives are polynomially bounded; for example,
\citep[Theorem~14.5.1]{MR1214374}.

In some cases
\citep{MR2214851,MR2783188}, it is possible to find a second SDE,
called the modified SDE with solution $\vec X_h(t)$, such
that $\vec X_n$ is a second-order weak approximation to
$\vec X_h(t)$; that is,
\begin{equation}\label{eq:aa}
\sup_{0\le t_n\le T}\mean{\phi(\vec X_h(t_n))}%
-\mean{\phi(\vec X_n)}%
=\order{\tstep^2}.
\end{equation}
Then, the solution of the modified equation $\vec X_h(t)$
is an order of $h$ closer to the numerical solution
than $\vec X(t)$.  The modified equation takes the form
\begin{equation}\label{eq:mod}
d\vec X_h%
=\tvecf(\vec X_h)\,dt+ \tG(\vec X_h)\,d\vec W(t),\qquad%
\vec X(0)%
=\vec X_0,
\end{equation}
where $\tvecf=\vec f+\tstep \vec f_1$ and $\tG=G+\tstep G_1$
for some $\vec f_1\colon\real^d \to \real^d$ and $G_1\colon
\real^d\to\real^{d\times m}$. This reduces to \cref{eq:sde}
with $h=0$, and $\vec f_1$ and $G_1$ describe the correction
in the drift and diffusion needed to achieve
\cref{eq:aa}. Our results concern SDEs and numerical
integrators where the second-order modified equation is
available. Except in special cases (e.g., if $G$ is
independent of $\vec X$), this does not include the
Euler--Maruyama method \citep{MR2214851}. It does include
the Milstein method, which has a second-order modified
equation \citep{MR2783188}.  Using weak-approximation theory
and modified equations, we develop an alternative method of
analysis for MLMC in this paper.  {By doing this, we no
longer depend {directly} on the strong-approximation properties of the
integrator (as in other papers, e.g. \citep{MR2436856}) and
this gives greater freedom {in the application} of MLMC.}

We {focus} on a class of integrators for an
important model in molecular dynamics and atmospheric
dispersion, the {\em Langevin equation}:
\begin{gather}
\begin{split}
  d\vec P%
  &=-\lambda \vec P\,dt%
  - \nabla V(\vec Q)\,dt%
  + \sigma\, d\vec W(t),\\
  d\vec Q%
 & =\vec P\,dt
\end{split}\label{eq:langevin}
\end{gather}
for parameters $\lambda,\sigma>0$, a potential $V\colon
\real^d \to\real$, and a $d$-dimensional vector $\vec W(t)$
of \iid Brownian motions. We specify initial conditions
$(\vec Q(0), \vec P(0))=(\vec Q_0, \vec P_0)\in\real^{2d}$.
This system is used in molecular dynamics to simulate a
system of particles in a heat bath and has equilibrium
distribution with pdf $Z^{-1}\exp(-H(\vec Q, \vec P)/ \kB
T)$, known as the \emph{Gibbs canonical distribution}, where
$Z$ is a normalisation constant, $H(\vec Q, \vec P)\coloneq
\tfrac12\vec P^\trans \vec P+V(\vec Q)$, and $\kB T
=\sigma^2/2\lambda$. As usual, $\kB$ denotes the Boltzmann
constant and $T$ temperature.  {The Langevin equation is
  also used to model the dispersion of atmospheric
  pollutants in homogeneous turbulence
  \citep{Rodean1996}. In that case, $\lambda$ is the inverse
  velocity autocorrelation time and
  $\sqrt{\sigma^2/2\lambda}$ is the strength of turbulent
  velocity fluctuations, and $d$ is equal to the number of
  space dimensions.  This is much smaller than in molecular
  dynamics applications, where $d$ is proportional to the
  number of particles. With a slight generalisation, it can
  also be used to model the dispersion in inhomogeneous
  turbulence.}

Numerical integrators for the
Langevin equation are well developed for example in
\citep{brunge84,wang2003analysis,cool}. Recently, there has
been a strong push to understand the invariant measure
associated to the integrators
\citep{kopec13:_weak_langev,MR2970763,MR2783188,MR2608370,leimkuhler14,andulle14:_long_lie,r1,MR3040887}.
Second-order modified equations are available for the most
important integrators for the Langevin equation. In
particular, we study splitting methods based on exact
sampling of an Ornstein--Uhlenbeck process and symplectic
integrators (symplectic Euler and St\"ormer--Verlet) for the
Hamiltonian part. {We show how to couple the
  different levels and apply MLMC with these methods. We
  find the use of the exact Ornstein--Uhlenbeck process is
  particularly effective when $\lambda$ is large.} %

{We also combine these new integrators with} extrapolation
\citep{talay:1990}.  It is a natural addition to MLMC
methods, already mentioned in the original work
\citep{MR2436856} and studied in more detail in
\citep{LemairePages2013}. It reduces the bias in the
numerical approximation of the solution due to time stepping
and relies on having a sharp estimate for the bias
error. If such an estimate is available, it is possible to
eliminate the leading-order error term in the bias error by
extrapolating from a sequence of approximations with
differing time-step sizes. These approximations are naturally
available in MLMC.

We provide a set of {experiments for the Langevin equation with a 
harmonic and a double-well potential, comparing integrators
based on splitting methods and extrapolation within MLMC.}
Our results confirm that the splitting methods are
significantly more effective than the Euler--Maruyama method
when combined with MLMC. All methods have the same
asymptotic $\epsilon$-cost; that is, the cost always grows
inverse proportionally to the mean-square error, but the
proportionality constant is reduced by an order of magnitude
from the standard Euler--Maruyama method through our
enhancements.

{Finally, we show how discrete random variables, as an
  approximation to the Gaussian increments of a Brownian
  motion, can be used within MLMC. This would be difficult
  to analyse by the standard analysis, since all the
  approximation results for integrators based on discrete
  random variables are in distribution only
  (e.g.,~\citep[\S14.2]{MR1214374}). In general, one must be
  careful in using discrete random variables in place of
  Gaussian random variables. The discrete approximations do not share
  the property of Gaussian random variables that the sum of two
  independent increments is itself an increment from the same
  distribution and hence the telescoping sum property, which is key to the
  standard MLMC idea, no longer holds. However, for a
  practical range of parameters in small-noise problems, the
  extra bias introduced is small and easy to estimate. Accepting this
  extra bias can lead to a significant improvement in efficiency,
  since discrete random variables allow the
  exact evaluation of the expected value on the coarsest level. The cost
  of this direct evaluation grows exponentially with the
  number of time steps, but it requires no sampling and, for a
  small number of time steps, its cost is significantly
  smaller than that of a Monte Carlo estimator. To analyse this
  method, we prove a new complexity theorem that allows for
  extra bias to be introduced between levels in MLMC. 

}

The paper is organised as follows. \cref{sec:mlmc} reviews
MLMC, including the important complexity
theorem. \cref{sec:anal} uses modified equations to apply
the complexity theorem, depending only on weak convergence
of the integrators. \cref{sec:geom} reviews splitting
methods for the Langevin equation and defines a number of
integrators where modified equations are available.
Numerical experiments are presented in \cref{sec:num_exp} to
demonstrate the effectiveness of this methodology for the
Langevin equation and to give quantitative predictions of
the possible gains.  {A final section considers approximation
of the Gaussian increments by discrete random variables and
highlights the potential gains this can bring.}
The C++ source code that we developed for the numerical experiments is freely available for download under the LGPL 3 license.

\section{Background on MLMC}\label{sec:mlmc}

When solving an SDE numerically, the total error consists of the
bias due to the time-stepping method and the Monte Carlo sampling
error. The sum of these two terms should be reduced below a given
small tolerance $\epsilon$. A standard Monte Carlo method
achieves this by computing $\Np$ sample paths, with
$\Np^{-1}=\order{\epsilon^{2}}$, and taking time step
$h=\mathcal{O}\bigl(\epsilon^{1/\alpha}\bigr)$, where $\alpha$ is the order of
weak convergence (e.g., $\alpha=1$ for the Euler--Maruyama
method). Hence, we can achieve accuracy $\epsilon$ with total
cost $\Cost{MC}(\epsilon)=\order{h^{-1}\times
  \Np}=\mathcal{O}\bigl(\epsilon^{-(2+1/\alpha)}\bigr)$. In contrast, MLMC uses a
series of coarse levels with larger time steps to construct an
estimator. 
{If the strong order of convergence of the employed
  integrator is one, MLMC reduces} the cost of the method to
$\Cost{MLMC}(\epsilon)=\order{\Np}=\order{\epsilon^{-2}}$,
which is the lower limit for 
a Monte Carlo method.

While MLMC is more efficient than standard Monte Carlo in
the limit $\epsilon\rightarrow0$, the actual value of the
tolerance $\epsilon$ might be relatively large in practical
applications. Hence, not only the asymptotic rate of
convergence, but also the cost of the method for a given
$\epsilon$ is of interest.  The exact value of the constant
$C_2$ in the cost function
$\Cost{MLMC}(\epsilon)=C_2\epsilon^{-2}+\cdots$ and the size
of higher-order corrections depends on the details of the
method, such as the time-stepping scheme and the coarse-level
solver. In particular, choosing a time-stepping scheme that
becomes unstable on the coarser levels can severely limit
the performance as only a small number of levels can be
used; see \citep{arnulf:2012, AbdulleBlumenthal2013}.

Suppose that we are interested in the expectation of
$\QOI\coloneq\phi(\vec X(T))$, where $\vec{X}(T)$ is the
solution to \cref{eq:sde} at time $T$ and
$\phi\colon\real^d\to \real$ defines the quantity of interest.
Assume that the number of time steps used to discretise the
SDE is $\Nt=\Ncoarse 2^L$, where $\Ncoarse, L\in
\naturals$. Our strategy is to approximate \cref{eq:sde}
using a numerical integrator with time step $h=T/\Nt$ to
define an approximate solution $\vec X_\Nt\approx \vec
X(T)$. Then, we compute many independent samples of $\vec
X_\Nt$ to define approximate samples $\QOI^{(i)}$ of
$\QOI$. The classical Monte Carlo method approximates
$\mean{\QOI}$ by the sample average of $\QOI^{(i)}$.

Instead, the MLMC method constructs a sequence of approximations
on levels indexed by $\ell=\{L,L-1,\dots,0\}$ with
$\Nt_\ell=\Ncoarse 2^\ell$ time steps of size $h_\ell=T/\Nt_\ell$.
Let $\QOI_\ell^{(i)}$ denote independent samples of the
approximation to $\mathcal{P}$ on level $\ell$ and let
\begin{equation}
  \Phat_{\ell} 
  \coloneq \frac{1}{\Np_\ell} \sum_{i=1}^{\Np_\ell} \QOI_\ell^{(i)}
\end{equation}
denote the Monte Carlo estimator on level $\ell$ based on
$\Np_\ell$ samples.  An estimator for the finest
level where $\Nt=\Nt_L$ can be written as the telescoping sum
\begin{equation}\label{eq22}
  {\Phat^{(\MLMC)}} \equiv
  \Yhat_{\left\{\Np_\ell\right\}} %
  \coloneq 
  \sum_{\ell=0}^{L} \Yhat_{\ell,\Np_\ell}\;,\vspace{-1.5ex}
\end{equation}
 where $\Yhat_{0,\Np_{0}}\coloneq\Phat_{0}$ and
\begin{equation}\label{eq:pp}
  \Yhat_{\ell,\Np_\ell} %
  \coloneq \frac{1}{\Np_\ell} \sum_{i=1}^{\Np_\ell}
  Y_\ell^{(i)},\qquad %
  Y_\ell^{(i)}\coloneq \QOI_\ell^{(i)}-\QOI_{\ell-1}^{(i)},
  \quad\text{for $\ell \ge 1$.}
\end{equation}
 The estimator does not introduce any additional bias,
as we recover the numerical discretisation error on the
finest level (where $h=h_L$):
\begin{equation}
  \expect{\Phat^{(\MLMC)}} =
  \expect{\Yhat_{\{\Np_\ell\}}} =
  \expect{\Phat_L} =
  \expect{\Phat^{(\MC)}},
\end{equation}
where $\Phat^{(\MC)}$ is the standard Monte Carlo estimator for
$\Nt=\Nt_L$ time steps. The two key ideas of the MLMC method are
now:
\begin{itemize}
\item The number of time steps $\Nt_\ell$ is smaller on the
  coarser levels $\ell<L$. Hence, the calculation of a
  single sample $\QOI_\ell^{(i)}$ is substantially cheaper.
\item The success of the method depends on coupling the
  samples $\QOI_\ell^{(i)}$ and $\QOI_{\ell-1}^{(i)}$ so
  that the variance of $Y_\ell^{(i)} =
  \QOI_{\ell}^{(i)}-\QOI_{\ell-1}^{(i)}$ is small. By
  arranging for the variance of $Y^{(i)}_\ell $ to be small,
  a smaller number $\Np_\ell$ of samples suffices to
  construct an accurate estimator
  $\Yhat_{\ell,\Np_\ell}$. This allows the construction of a
  MLMC estimator with fixed total variance
  $\sum_{\ell=0}^L
  \operatorname{Var}\big[\Yhat_\ell\big]/\Np_\ell$ and lower
  computational cost.
\end{itemize}

This is formalised in the following complexity theorem
\citep[Theorem 3.1]{MR2436856}:
\begin{theorem}[MLMC complexity]\label{theo:ComplexityTheorem} %
  Consider a real-valued random variable $\QOI$ and estimators
  $\Phat_{\ell}$ corresponding to a numerical
  approximation based on time step $h_\ell=T/ \Nt_{\ell}$ and
  $\Np_\ell$ samples. If there exist independent estimators
  $\Yhat_{\ell,\Np_\ell}$ based on $\Np_\ell$ Monte Carlo
  samples, and positive constants $\alpha\ge\frac{1}{2}$,
  $c_1$, $c_2$, $c_3$ such that
\begin{enumerate}[\normalfont(i) ]
  \item $\vpair*{\expect{\Phat_\ell-\QOI}}\le c_1 h_{\ell}^\alpha$,
  \item $\expect{\Yhat_{\ell,\Np_\ell}}%
    =\begin{cases}\expect{\Phat_{\ell}}, & \ell=0,\\[0.5em]%
      \expect{\Phat_\ell-\Phat_{\ell-1}}, & \ell>0,\end{cases}$
  \item $\variance{\Yhat_{\ell,\Np_\ell}}\le c_2 \Np_\ell^{-1}h_\ell^2$,
    and
  \item $\Cost{MLMC}_\ell$, the computational complexity of
    $\Yhat_{\ell,\Np_\ell}$, is bounded by $c_3 \Np_\ell
    h_\ell^{-1}$,
\end{enumerate}
then there exists a positive constant $c_4$ such that for any
$\epsilon<1/e$, there are values $L$ and $\Np_\ell$ for which
$\Yhat_{\{\Np_\ell\}}$ from \cref{eq22}
has a mean-square error (MSE) with bound
\begin{equation}\label{end}
  \text{MSE} \equiv \expect{\left(\Yhat_{\{\Np_\ell\}}-\expect{\QOI}\right)^2} < \epsilon^2
\end{equation}
and a computational complexity $\Cost{MLMC}$ with bound
\begin{equation}
  \Cost{MLMC} \le c_4 \epsilon^{-2}.
\end{equation}
\end{theorem}
 
The theorem can be extended to allow the variance to decay
as $\var{\Yhat_\ell}\le c_2 \Np_\ell^{-1}h_\ell^\beta$
\citep{MR2436856}. For all cases in this paper, $\beta=2$
and the cost is concentrated on the coarsest level (as we
see from \cref{alg} line \ref{eqn:Nestimator} and (iii) above). The
asymptotic dependence of the computational complexity on
$\epsilon$ is independent of the weak order of convergence
$\alpha$ of the time-stepping method. However, the constant
$c_4$ does depend on the particular time-stepping method.

To obtain the results in this paper, we used Algorithm
\ref{alg:MLMC} and our choices for the numbers of samples
$\Np_\ell$ on each of the levels are defined
{adaptively via $\Np_\ell^+$ using the sample
  variance following \citep{MR2436856}}.  Given a tolerance
$\epsilon_{\max}>0$, the algorithm gives an MLMC estimator
$\Phat^{\text{(MLMC)}}$ with mean-square error $\epsilon$ in
the range $\epsilon_{\max}/2 < \epsilon < \epsilon_{\max}$
as defined in
\cref{end}.
\begin{algorithm}
  \caption{Multilevel Monte Carlo. Input: 
    $\epsilon_{\max}$, $\Ncoarse$,  and $T$. Output:
    Estimator $\Phat^{(\MLMC)}$\label{alg:MLMC}}\label[algorithm]{alg}
 \begin{algorithmic}[1]
   \STATE Choose $L, \Ncoarse$ such that, on the finest level
   with $\Nt_L=2^L\Ncoarse$ time steps of size $h_L=T/\Nt_L$, the
   bias $\epsilon_{\operatorname{bias}}$ is smaller than
   $\epsilon_{\max}/\sqrt{2}$. Define $\epsilon\equiv
   \sqrt{2}\epsilon_{\operatorname{bias}}$.  
   \STATE Choose a minimum number of samples $\Np_{\min}$ (say
   $100$ or $1000$).  
   \STATE Set $\Np_\ell^-=1$,
   $\Np_\ell^+=\Np_{\min}$, $\Np_{\ell}=0$ for all levels $\ell$.
   \WHILE{$\Np_\ell<\Np_\ell^+$ for some level $\ell$} \FOR
   {$\ell=L,\dots,0$} \FOR
   {$i=\Np_\ell^-,\dots,\Np_\ell^+$} 
   \STATE Calculate
   $Y_\ell^{(i)}$ by applying the numerical integrator on
   levels $\ell$ and $\ell-1$ {(except for $\ell=0$) for
     sample $i$. The two trajectories
   should be coupled (see~\cref{ss:c}), but $Y_\ell^{(i)}$ should be independent of 
   any other sample (i.e., of $Y_{\ell'}^{(i')}$ for
   $\ell'\not=\ell$ or $i'\ne i$).}
   \STATE $\Np_\ell \mapsto \Np_\ell+1$.
      \ENDFOR 
      \STATE       \label{eq:algEstimators} Update estimators for the bias and variance:
      \[
          \Yhat_{\ell,\Np_\ell} = \frac{1}{\Np_\ell}\sum_{i=1}^{\Np_\ell} Y_\ell^{(i)},\quad
          \Vhat_{\ell,\Np_\ell} = \frac{1}{\Np_\ell-1}\left[
            \sum_{i=1}^{\Np_\ell}\left(Y_\ell^{(i)}\right)^2
            - \frac{1}{\Np_\ell}\left(\sum_{i=1}^{\Np_\ell} Y_\ell^{(i)}\right)^2
          \right].
      \]
      \STATE $\Np_{\ell}^-=\Np_{\ell}^++1$.
      \STATE \label{eqn:Nestimator} Calculate the optimal $\Np_\ell^+$ according to formula (12) in \citep{MR2436856}:
      \[
        \Np_\ell^+ =
        \ceil*{2\epsilon^{-2}\sqrt{\Vhat_{\ell,\Np_\ell}
            h_\ell}\left(\sum_{j=0}^L\sqrt{\Vhat_{j,\Np_j}/h_j}\right)}.%
              \]
    \ENDFOR
  \ENDWHILE
  \STATE Return estimator $\Phat^{(\MLMC)} = \Yhat_{\{\Np_\ell\}} \equiv \sum_{\ell=0}^L \Yhat_{\ell,\Np_\ell}$ 
 \end{algorithmic}
\end{algorithm}


\section{Applying the complexity theorem}\label{sec:anal}

Our goal is to apply the complexity theorem to numerical
integrators using only weak-approximation properties of the
numerical methods. The
complexity theorem makes assumptions on (i) the bias, (ii) the
consistency of the estimators, and (iii) the variance of the
corrections. (i) can be understood from existing weak-convergence
analysis. Let $\Cpinfty(\real^d)$ be the set of infinitely
differentiable functions $\real^d\to \real$ such that all
derivatives are polynomially bounded.
\begin{definition} For a time step $h>0$, let $\vec X_n$ be a
  $\real^d$-valued random variable that approximates the solution
  $\vec X(t)$ to \cref{eq:sde} at time $t=nh$.  We say 
  $\vec X_n$ is a \emph{weak order-$\alpha$} approximation if, for all
  $\phi\in\Cpinfty(\real^d)$ and $T>0$, there exists $K>0$
  such that for $h$ sufficiently small
  \[
  \abs{ \mean{\phi(\vec X(T))}%
    -\mean{\phi(\vec X_\Nt)}} %
  \le K h^\alpha,\qquad \Nt h =T.
  \] 
\end{definition}
There are many integrators that provide weak order-$\alpha$
approximations for $\alpha=1$ or $\alpha=2$ (e.g.,
\cite{MR1214374} or \cref{sec:geom}). In the case that
$\mathcal{P}=\phi(\vec X(T))$, $T=\Nt h_\ell$, and
$\Phat_\ell=\phi(\vec X_\Nt)$ for an $\vec X_\Nt$ with step
$h_\ell$ that is weak $\alpha$-order, the bias condition (i)
holds.

The consistency of the estimators (ii) is an easy consequence of
the linearity of integration and \cref{eq:pp}.

Condition (iii) on the variance of corrections normally
follows from the mean-square convergence of the integrator
\citep{MR2436856}. Mean-square convergence measures the
approximation of individual sample paths of the solution
$\vec X(t)$ and hence is a tool for understand the coupling
of successive levels. In this paper, we use an alternative
method based on weak-approximation theory and derive
condition (iii) as a consequence of the existence of a
second-order modified equation. To do this, we introduce the
following doubled-up system for $\vec Z=[\vec X, \vec
Y]\in\real^{2d}$:
\begin{gather}
  \begin{split}
  d\vec X&=\vec f(\vec X)\,dt+ G(\vec X)\,d\vec W(t), \qquad \vec
X(0)=\vec X_0\in\real^d,\\
  d\vec Y&=\vec f(\vec Y)\,dt+ G(\vec Y)\,d\vec W(t), \qquad\; \vec
Y(0)=\vec X_0.
\end{split}\label{eq:du}%
\end{gather}
The same initial data is applied and the same $\vec W(t)$
drives both components and so $\vec X(t)=\vec Y(t)$ a.s. for
$t>0$. We now have two copies of $\vec X(t)$ and we
approximate each differently.  Formally, we approximate
$\vec X(t)$ and $\vec Y(t)$ by different numerical
integrators with step $h>0$ and denote the resulting
approximation to $\vec Z(t_n)$ by $\vec Z_n=[\vec X_n, \vec
Y_n]$ at $t_n=n h$. In MLMC, there is usually one integrator
applied with time steps $\tstep$ for $\vec X$ and $\tstep/2$
for $\vec Y$ (which is a little awkward for $\vec Y_n$, as
one increment of $n$ corresponds to two steps of the
underlying integrator).  The joint distribution of the $\vec
X_n$ and $\vec Y_n$ contains all the required information
about the coupling of the approximations of each component
and, as we now show, a weak-convergence analysis of the
system gives condition (ii).

For simplicity, we start by assuming that $\vec Z_n=[\vec X_n,
\vec Y_n]$ is a weak second-order approximation to $\vec
Z(t)=[\vec X(t), \vec Y(t)]$. Then, we can prove the following.

\begin{theorem}\label{maint1} Fix $T>0$ and let $\mathcal{P}=\phi(\vec
  X(T))$ for a $\phi\in \Cpinfty(\real^d)$. Suppose that
  $\vec Z_n$ is a weak second-order approximation to $\vec
  Z(t)$. Conditions (i)--(iii) of
  \cref{theo:ComplexityTheorem} hold with
  $[\QOI_\ell^{(i)},\QOI_{\ell-1}^{(i)}]$ given by \iid
  samples of $[\phi(\vec X_\Nt), \phi(\vec Y_\Nt)]$ with
  $h=h_\ell$ and $\Nt h=T$.
\end{theorem}

\begin{proof}
  The condition on $\vec Z_n$ implies also that $\vec X_n$
  and $\vec Y_n$ are weak second-order approximations to
  $\vec X(t)$.  Then, by the above discussion, conditions
  (i) with $\alpha=2$ and (ii) hold.

  Let $\psi(\vec Z)\coloneq\phi(\vec X)-\phi(\vec Y)$. Then
  $\psi^2 \in \Cpinfty(\real^{2d})$ since $\phi$ and hence
  $\psi^2$ are smooth and their derivatives are polynomially
  bounded. As $\vec Z_n$ is a weak second-order
  approximation to $\vec Z(t)$,
  \[
  \mean{\psi(\vec Z_\Nt)^2-\psi(\vec Z(T))^2}=\order{h^2}.
  \] 
  By definition of $\psi$,
  \begin{equation}\label{eq:b}
    \mean{\abs{\phi(\vec X_\Nt)-\phi(\vec Y_\Nt)}^2%
      -\abs{\phi(\vec X(T))-\phi(\vec Y(T))}^2}=\order{h^2}.
  \end{equation}
  Using the fact that $\vec X(t) = \vec Y(t)$ a.s., we have  \[
  \mean{\abs{\phi(\vec X_\Nt)-\phi(\vec Y_\Nt)}^2}=\order{h^2}.
  \] 
  Written in terms of $\mathcal{P}_\ell$ and
  $\mathcal{P}_{\ell-1}$, this means
\[
\var\bp{\mathcal{P}_\ell-\mathcal{P}_{\ell-1}}%
\le \mean{\pp{\mathcal{P}_\ell-\mathcal{P}_{\ell-1}}^2}%
=\order{h^2}.
  \] 
  In other words, the variance of each sample of the
  coarse--fine correction is order $\tstep^2$. This implies
  that the sample average $\Yhat_{\ell,\Np_\ell}$ of $\Np_\ell$
  \iid samples satisfies condition (iii) of
  \cref{theo:ComplexityTheorem}.
\end{proof}
\subsection{Modified equations}
The above argument does not apply to weak first-order
accurate methods, even though the complexity theorem only
requires $\alpha>1/2$. In this case, we use the theory of
modified equations to extend the analysis. A modified
equation is a small perturbation of the original SDE that
the numerical method under consideration approximates more
accurately. For the theory, we need a second-order modified
equation for the doubled-up system and this contains
second-order information about the coupling of the fine and
coarse levels. In particular, we consider modified equations
for the double-up system \eqref{eq:du} of the form:
\begin{gather}
  \begin{split}
    d\vec X_h&=\bp{\vec f(\vec X_h)+h \vec f_1(\vec
      X_h)}\,dt%
    + \bp{G(\vec X)+h G_1(\vec X_h)}\,d\vec W(t), \qquad
    \vec X(0)%
    =\vec X_0,\\
    d\vec Y_h&=\bp{\vec f(\vec Y_h)%
      +h \vec f_2(\vec Y_h)}\,dt%
    + \bp{G(\vec Y_h)+h G_2(\vec Y_h)}\,d\vec W(t), \qquad
    \vec Y(0)=\vec X_0,
\end{split}\label{eq:modeq}%
\end{gather}
for $\vec f_i\colon \real^d\to\real^d$ and $\vec
G_i\colon\real^d\to\real^{d\times m}$ for $i=1,2$. (This
could be extended to allow $\vec f_i, G_i$ to depend on both $\vec
X_h$ and $\vec Y_h$.) When the same integrator is used for
each component, but with time steps $h$ and $h/2$, it must
hold that $\vec f_2=\vec f_1/2$ and $G_2=G_1/2$. 
We  show in \cref{maint} that, subject to regularity conditions on the
coefficients, the MLMC complexity theorem applies if a
second-order modified equation exists and therefore MLMC works
with $\order{\epsilon^{-2}}$ complexity.

The additional difficulty is that $\vec X_h\ne \vec Y_h$ and
we must estimate the variance of $\phi(\vec X_h)-\phi(\vec
Y_h)$. We use a mean-square analysis and the following
lemma, which gives a first-order $L^2(\Omega,\real^d)$ bound
on $\vec Z(t)-\vec Z_h(t)$.  The lemma requires a number of
regularity assumptions on the coefficients of the modified
equation, which hold, for example, if $\vec f,\vec f_i$ and
$G,G_i$ are globally Lipschitz continuous.

\begin{lemma}\label[lemma]{lemma:1}
  For $t\in [0,T]$, let $\vec Z(t)$ satisfy the Ito SDE
  \eqref{eq:du} and $\vec Z_h(t)=[\vec X_h(t), \vec Y_h(t)]$
  satisfy the modified equation \eqref{eq:modeq}.  Suppose
  that
\begin{enumerate}[\normalfont(i) ]
\item   $\vec f\colon \real^d\to \real^d$ and $G\colon \real^d\to
  \real^{d\times m}$ are globally Lipschitz continuous with
  Lipschitz constant $L>0$.
\item There exists $C_1>0$ such that, for all $h>0$
  sufficiently small,
  \begin{align*}
    \mean{\norm{\vec f_1(\vec X_h(s))}^2},%
    \mean{\norm{G_1(\vec X_h(s))}_\rmF^2}%
    &\le C_1,\\
    \mean{\norm{\vec f_2(\vec Y_h(s))}^2},%
    \mean{\norm{G_2(\vec Y_h(s))}_\rmF^2}%
    &\le C_1,\qquad s\in [0,T],
  \end{align*}
  where $\norm{\cdot}_\rmF$ denotes the Frobenius norm.
\end{enumerate}
Then, if $\psi\colon \real^{2d}\to \real$ is globally Lipschitz
continuous, we have, for some constant $C_2>0$ independent of
$h$, \[
\mean{\abs{\psi(\vec Z(t))-\psi(\vec Z_h(t))}^2}\le C_2
h^2,\qquad\text{ for $t\in [0,T]$.}
\]
\end{lemma}
\begin{proof}
  This is an elementary calculation with the Gronwall inequality
  and Ito isometry. See \cref{useful_lem}.
\end{proof} 
We are now able to state and prove the main theorem of this
article. {In contrast to \cref{maint1}, $\phi$ is
  assumed to be Lipschitz here.}

\begin{theorem}\label{maint} Fix $T>0$. Let  $\phi\in
  \Cpinfty(\real^d)$ be globally Lipschitz continuous. Suppose
  that
  \begin{enumerate}[\normalfont(i) ]
  \item $\vec X_n$ and $\vec Y_n$ are weak order-$\alpha$
    approximations to $\vec X(t)$ for some $\alpha>1/2$,
  \item $\vec Z_n$ are second-order weak approximations to $\vec
    Z_h(t)$, and
  \item the assumptions of \cref{lemma:1} hold.
  \end{enumerate}
  Then Conditions (i)--(iii) of \cref{theo:ComplexityTheorem}
  hold with $[\QOI_\ell^{(i)},\QOI_{\ell-1}^{(i)}]$ given by \iid
  samples of $[\phi(\vec X_\Nt), \phi(\vec Y_\Nt)]$ with $h=h_\ell$.
\end{theorem}

\begin{proof}
  As before, conditions (i) and (ii) are straightforward.
  It is the third condition, which normally follows from a
  strong-approximation theory, that requires the modified
  equation.  Let $\psi(\vec Z)\coloneq\phi(\vec X)-\phi(\vec
  Y)$ and note that $\psi^2\in \Cpinfty(\real^{2d})$.  As
  $\vec Z_n$ is a second-order weak approximation to $\vec
  Z_h(t)$, we have
  \[
  \mean{\psi(\vec Z_\Nt)^2-\psi(\vec Z_h(T))^2}=\order{h^2}.
  \] 
  By definition of $\psi$,
  \begin{equation}\label{eq:b}
    \mean{\abs{\phi(\vec X_\Nt)-\phi(\vec Y_\Nt)}^2%
      -\abs{\phi(\vec X_h(T))-\phi(\vec Y_h(T))}^2}=\order{h^2}.
  \end{equation}
  Using the fact that $\vec X(t) = \vec Y(t)$ a.s.,
  \begin{equation*}
    \begin{aligned}
      \mean{\abs{\phi(\vec X_h(T))-\phi(\vec Y_h(T))}^2}%
      &= \mean{\abs{\phi(\vec X_h(T))%
          -\phi(\vec X(T))%
          +\phi(\vec Y(T))-\phi(\vec Y_h(T))}^2} \\%
      &= \mean{\abs{\psi(\vec Z_h(T))-\psi(\vec Z(T))}^2}.
    \end{aligned}
  \end{equation*}
  \cref{lemma:1} applies and the right-hand side in the last
  equation is $\order{h^2}$. Consequently,
  \begin{equation}
    \mean{\abs{\phi(\vec X_h(T))-\phi(\vec Y_h(T))}^2}
    =\order{h^2}.
    \label{eq:a}
  \end{equation}
  Together, \cref{eq:a,eq:b} imply that
  \begin{equation}
  \mean{\abs{\phi(\vec X_\Nt)-\phi(\vec Y_\Nt)}^2}=\order{h^2}.
\label{eq:c}
  \end{equation}
The remainder of the proof is the same as for \cref{maint1}.
\end{proof}

{By taking $\vec X$ to be the exact solution (i.e.,
  $\vec X_n=\vec X(t_n)$) and $\phi\colon \real^d\to \real$
  as a projection onto the $i$th coordinate, \cref{eq:c}
  implies that
  $\norm{\vec{X}(T)-\vec{Y}_\Nt}_{L^2(\Omega,\real^d)}=\order{h}$
  and hence first-order strong convergence can be proved by
  this method. This is consistent with the observation that
  the Euler--Maruyama method, which is not first-order
  strongly convergent in general, does not have a
  second-order modified equation.}

In summary, subject to smoothness conditions, if MLMC is
applied with an integrator that has a second-order modified
equation like \cref{eq:modeq} then the variance of the
coarse--fine correction is $\order{h^2}$ and the complexity
of MLMC is $\order{\epsilon^{-2}}$. Though the rate is
fixed, the complexity of MLMC depends on the specific
integrator used through the constant and, as we now show,
this leads to large variations in efficiency.

\section{Application to the Langevin equation}\label{sec:geom}
Before showing how they can be used for MLMC, we introduce
several integrators for the Langevin equation.
\subsection{Splitting methods}\label{sec:sm}

{Splitting methods} are an important class of numerical
integrators for differential equations. In the case of ODEs, they
allow the vector field to be broken down into meaningful parts
and integrated separately over a single time step, before
combining into an integrator for the full vector field. See for
example~\citep{MR2840298,MR2132573}. The Langevin equation breaks
down into the sum of a Hamiltonian system and a linear SDE for an
Ornstein--Uhlenbeck (OU) process.  Then, for a splitting method,
we define symplectic integrators for the Hamiltonian system
\begin{gather}
  \begin{split}
    \frac{d\vec Q}{dt}%
    &= \vec P,\\
    \frac{d\vec P}{dt}%
    &= -\nabla V(\vec Q).
  \end{split}\label{eq:ham}
\end{gather}
The OU process $\vec P(t)$, which satisfies
\begin{equation}
  d\vec P%
  =-\lambda \vec P\,dt%
  +\sigma \,d\vec W(t),\label{eq:ou}
\end{equation}
 can be integrated exactly and we use this fact to
define a so-called geometric integrator for \cref{eq:ou}.  It is
clear that the sum of the right-hand sides of these two systems
gives \cref{eq:langevin}.  There are a number of ways of
combining integrators of \cref{eq:ham,eq:ou} to define an
integrator of the full system.  The simplest, also known as 
the Lie--Trotter splitting, is to simulate \cref{eq:ham,eq:ou}
alternately on time intervals of length $\tstep$.  In general,
this technique can only be first-order accurate in the weak
sense. Alternatively, if the underlying integrators are second
order, we can define a second-order splitting method by applying
\cref{eq:ou} on a half step, then \cref{eq:ham} for a full step,
and finally apply again \cref{eq:ou} on a half step. This is
called the symmetric Strang splitting. {See also \citep{r1}.}

We now define specific integrators for \cref{eq:ham,eq:ou}.
\cref{eq:ham} is a separable Hamiltonian system, and the
symplectic Euler method and St\"ormer--Verlet methods
provide {simple, explicit methods} for its numerical
solution. The symplectic Euler method is first-order
accurate and the St\"ormer--Verlet method is second-order
accurate.

The solution of \cref{eq:ou} is a multi-dimensional OU process
and can be written as
\begin{equation}\label{eq:voc}
\vec P(t)%
=\erm^{-\lambda t} \vec  P(0)+\sigma \vec I(0,t),\qquad
\vec I(t_1,t_2)%
\coloneq\int_{t_1}^{t_2} \erm^{-\lambda(t_2-s)}\,d\vec W(s).
\end{equation}
Each component of $\vec I$ is \iid with mean zero and variance
\begin{equation}\label{eq:geo_var}
  \var{I_i(t_1,t_2)}%
  =\int_{t_1}^{t_2} \erm^{-2\lambda(t_2-s)}\,ds%
  = 
  \frac{1-{e}^{- 2\, \lambda\, (t_2-t_1)} }{2\, \lambda},
\end{equation} 
so that $\vec I(t_1,t_2)\sim \Nrm(\vec 0, \alpha^2_{t_2-t_1} I)$ for
$\alpha_t\coloneq \sqrt{(1-\erm^{-2\lambda t})/2\lambda}$.  This
suggests taking the following as the numerical integrator: for a
time step $\tstep>0$,
\begin{equation}\label{eq:geom}
\vec P_{n+1}%
=\erm^{-\lambda \tstep}\vec P_n %
+ \sigma \alpha_\tstep \vec\xi_n
\end{equation}
for $\vec \xi_n\sim \Nrm(0,I)$ \iid. If $\vec P_n=\vec
P(t_n)$, then $\vec P_{n+1}$ has the same distribution as
$\vec P(t_{n+1})$ and this method is exact in the sense of
distributions. Methods of this type, where the variation of
constants formula~\eqref{eq:voc} is used for the
discretisation, are often called geometric
integrators \citep{MR2491434}.  

The full equations for the first order splitting (symplectic
Euler) and second-order splitting (St\"ormer--Verlet) are
written as follows:
\paragraph{Symplectic Euler/OU} For $\vec \xi_n$ \iid with
distribution $\Nrm(\vec 0,I)$,
\begin{gather}
    \begin{split}
      \vec P^*_{n+1}%
     & =\erm^{-\lambda \tstep}\vec P_n %
      + \sigma \alpha_\tstep \vec\xi_n,\\
    \vec P_{n+1}%
    &= \vec P_{n+1}^* - h\,\nabla V(\vec Q_n),\\
    \vec Q_{n+1}%
    &= \vec Q_n + \vec P_{n+1}\,\tstep.
    \end{split}\label{eq:y}
  \end{gather}
  \paragraph{St\"ormer--Verlet/OU}For $\vec\xi_n,\vec\xi_{n+1/2}$
  \iid with distribution $\Nrm(\vec 0,I)$
\begin{gather}
    \begin{split}
      \vec P^*_{n+1/2}%
   &   =\erm^{-\lambda \tstep/2}\vec P_n %
      + \sigma \alpha_{\tstep/2} \vec\xi_n,\\
    \vec P_{n+1/2}&=\vec P^*_{n+1/2} - \frac 12 h\, \nabla V(\vec Q_n),\\
    \vec Q_{n+1}&=\vec Q_{n} + h \vec P_{n+1/2},\\
    \vec P^*_{n+1}&=\vec P_{n+1/2} - \frac 12 h\, \nabla V(\vec Q_{n+1}),\\ 
    \vec P_{n+1}%
    &  =\erm^{-\lambda \tstep/2}\vec P^*_{n+1} %
      + \sigma \alpha_{\tstep/2} \vec\xi_{n+1/2}.
\end{split}\label{eq:yy}
\end{gather}
Subject to regularity conditions on the coefficients,
\cref{eq:y} is first-order and \cref{eq:yy} second-order
accurate in the weak sense by application of the
Baker--Campbell--Hausdorff formula.

\subsection{Modified equations for the Langevin equation}
Consider the Langevin equation \eqref{eq:langevin}.
{Following \citep{MR2214851,MR2783188} by using} a
computer algebra system to verify consistency of moments to fifth
order, it is easy to find modified equations for the numerical
integrators developed in \cref{sec:sm}.  For example, for the
first-order splitting method with $d=1$, the doubled-up modified
equation is as follows: Denote by $[Q_n,P_n]$ the numerical
approximation on the coarse level (step $h$) and $[q_n,p_n]$ on
the fine level (step $h/2$). The second-order modified equation
is
\begin{gather}\label{eq:mod1}
\begin{split}
  dQ&=\bp{ P -\frac 12 h \pp{V'(Q)+\lambda P}}\,dt+
  \sigma \frac12
  h \,dW(t),\\
  dP&=\bp{ -\lambda P - V'(Q)- \frac 12 h\pp{\lambda
      V'(Q)-P V''(Q)}}\,dt + \sigma \,dW(t),\\
  dq&=\bp{ p -\frac 14 h \pp {V'(q)+\lambda p}}\,dt+
  \sigma \frac14 h \,dW(t),\\
  dp&=\bp{ -\lambda p - V'(q)- \frac 14 h\pp{\lambda
      V'(q)-p V''(q)} }\,dt + \sigma \,dW(t),
\end{split}
\end{gather}
where $W(t)$ is the same Brownian motion for $p$ and $P$. We
conclude then that this method leads to $\order{h^2}$ variances
in the coarse--fine correction, if the coefficients are
sufficiently well behaved. Identifying when the coefficients are
well behaved is hard. For example, it is sufficient that the
drift and diffusion in both the original and modified equations
are globally Lipschitz. These however are very strong conditions
and do not hold for many realistic potentials.

For the second-order splitting method (based on
St\"ormer--Verlet method and exact OU integration), we can
apply \cref{maint1} to see that the variance of the
coarse--fine corrections is $\order{h^2}$.  The
regularity condition is on the original drift and diffusion
and holds if $\nabla V\colon \real^d\to\real^d$ is
sufficiently smooth (e.g., infinitely differentiable and
Lipschitz).

\subsection{MLMC with splitting methods}\label{ss:c}

Let $\vec X=[\vec Q, \vec P]$ denote the state-space variable.  A
key step in MLMC  is computing approximations {to $\vec
X(t_{n+2})$ at $t_n=n\tstep$ given $\vec X(t_n)$ based on
integrators with time steps $\tstep$ and $2\tstep$} that are
coupled so the difference between the approximations has small
variance. For the Euler--Maruyama method, this is achieved by
choosing increments $\Delta \vec W_n, \Delta \vec W_{n+1}$ for
the computation with time step $\tstep/2$, and choosing the sum
$\Delta \vec W_n+\Delta \vec W_{n+1}$ for the corresponding interval of the
computation with time step $\tstep$. 

It is hard to sample $\vec I(0,t)$ in \cref{eq:geom} based
on increments of the particular sample path of $\vec W(t)$
and, as a method for strong approximation, it is limited. It
is easy however to sample $\vec I(0,t)$ as a Gaussian random
variable. We now show how to couple fine--coarse integrators
for the MLMC method, without the direct link to the
increment. First, note that
\begin{align*}
  \vec I(0,2h)%
  &=\int_0^{2h} \erm^{-\lambda(2h-s)}\,d\vec W(s)\\%
  &=\int_0^{h} \erm^{-\lambda(2h-s)}\,d\vec W(s)%
  +  \int_h^{2h} \erm^{-\lambda(2h-s)}\,d\vec W(s)\\
 &= r I(0,h)+I(h,2h),\qquad r\coloneq \erm^{-\lambda h}.
\end{align*}

\[
\vec I(0,h),\vec I(h,2h)\sim \Nrm\ppair*{0, \alpha_h^2 I}\; \text{\iid}.
\]

We can simulate $\vec I(0,h)$ and $\vec I(0,2h)$, by generating
$\vec \xi_i\sim \Nrm(\vec 0,I)$ \iid and computing
\[
\vec I(0,h)%
=\alpha_h\vec \xi_1,\qquad
\vec I(h,2h)%
=\alpha_h\vec \xi_2.
\]
As $\alpha_h^2=(1-r^2)/2\lambda$ and
$\alpha_{2h}^2=(1-r^4)/2\lambda$, we have
$\alpha_h^2(1+r^2)=\alpha_{2h}^2$.
Then,
\begin{equation}\label{ut}
\vec I(0,2h)%
= \alpha_{2h}\frac1{\sqrt{1+r^2}}\ppair*{r \vec \xi_1+\vec \xi_2}. %
\end{equation}
Given $\vec P_n$ at time $t_n$, we find 
$\vec  P_{n+2}$ using two time steps of size $h$ by
\begin{align*}
  \vec P_{n+1}%
  &= \erm^{-\lambda \tstep} \vec P_n %
  + \sigma \alpha_{h}\vec \xi_n\\
  \vec P_{n+2}%
  &= \erm^{-\lambda \tstep} \vec P_{n+1} %
  + \sigma \alpha_{h} \vec \xi_{n+1},
\end{align*}
for $\vec \xi_n\sim \Nrm(0,I)$ \iid.
This is equivalent to  a single time step of size $2h$ and
\[
\vec P_{n+2}%
= \erm^{-2\lambda h} \vec P_n %
+ \sigma \alpha_{2h} \vec \xi^*_n,\qquad %
\vec \xi_n^{*}%
\coloneq \frac{r \vec \xi_n+\vec \xi_{n+1}}{\sqrt{r^2+1}}.
\]
This method is used to generate {the increments when
using splitting methods within MLMC.}

\section{Numerical experiments}\label{sec:num_exp}

We developed an object-oriented C++ code to compare the performance of different numerical methods for two model problems. The modular structure of the templated code makes it easy to change key components, such as the time-stepping method or random-number distribution, without negative impacts on the performance. The source code is available under the LGPL 3 license as a git repository on \verb!https://bitbucket.org/em459/mlmclangevin!\footnote{All enquiries about the code should be addressed to \texttt{e.mueller@bath.ac.uk}.}.

Key to the choice of parameters in Algorithm \ref{alg:MLMC} is the balance
between bias error and statistical error. We assume that the bias error
has the form in \cref{theo:ComplexityTheorem}(i) for a
proportionality constant $c_1$ and that the finest time step
$h_L=T/(\Ncoarse 2^L)$. Then, for a bias error of size
$\epsilon/\sqrt 2$, we require that
\[
c_1 \pp{ \frac T{\Ncoarse 2^L}}^\alpha%
= \frac{\epsilon}{\sqrt{2}}.
\]
Given $c_1$, $\alpha,$ and $T$ as well as a choice for
$\Ncoarse$, this can be solved to determine $\epsilon$ from
$L$ or vice versa.  
The constant $c_1$ can be
approximated by assuming that $\mean{\Phat_\ell-\mathcal
  P}=\tilde c_1 h_\ell^\alpha$ for some $\tilde c_1\in
\real$, so that
\begin{align*}
  \Yhat_{\ell,\Np_\ell}%
  \approx \tilde c_1 h_\ell^\alpha-\tilde c_1 h_{\ell-1}^\alpha%
  =\tilde c_1(1-2^{\alpha})h_\ell^\alpha
\end{align*}
and calculating $c_1=\abs{\tilde c_1}$ after computing the left-hand side
numerically.

{
The following integrators are used in the numerical experiments below: 
\begin{description}
\item[EMG and EMG+] Euler--Maruyama  as
  given by \cref{eq:1} with $\Ncoarse=4$ (EMG) 
and $\Ncoarse=8$ (EMG+).

\item[SEG] First-order splitting method with symplectic
  Euler/exact OU and $\Ncoarse=4$. See
  \cref{eq:y}.
\item[SVG] Second-order splitting method with
  St\"ormer--Verlet/exact OU and $\Ncoarse=4$. See
  \cref{eq:yy}.
\end{description}}
 Richardson extrapolation is a well-known
technique for increasing the accuracy of a numerical
approximation by computing two approximations with different
discretisation parameters and taking a linear combination that
eliminates the lowest-order term for the error. Its extension to
SDEs was developed by \citep{talay:1990} and is particularly
convenient for use with MLMC, as MLMC computes approximations on
several levels {and this has already been explored in \citep{MR2436856}.} Thus, we take $\Phat_{L}$ and $\Phat_{L-1}$ and
suppose that, for some constants $\tilde c_1$ and $\alpha'>\alpha$,
\begin{align*}
  \mean{\Phat_L}&=\mean{\mathcal P}%
  +\tilde c_1 h_L^\alpha%
  + \order{h_L^{\alpha'}}\\
  \mean{\Phat_{L-1}}%
  &=\mean{\mathcal P}%
  +\tilde c_1 h_{L-1}^\alpha%
  +\order{h_L^{\alpha'}}.
\end{align*}  
A simple linear combination of the two gives a higher-order
approximation to $\mean{\mathcal P}$; in particular, for SEG, we have
$\alpha=1$ and $\alpha'=2$ and
\[
2\mean{\Phat_{L}}-\mean{\Phat_{L-1}}%
=\mean{\mathcal P}+\order{h_\ell^{2}}.
\]
An approximation to the left-hand side is given by $\Phat^{(\MLMC)}+\Yhat_{L,\Np_L}$.
For SVG, we have $\alpha=2$ and $\alpha'=4$, and
\[
\frac{1}{3}\pp{4\mean{\Phat_{L}}-\mean{\Phat_{L-1}}}%
=\mean{\mathcal P}+\order{h_\ell^{4}}.
\]
An approximation to the left-hand side is given by $\Phat^{(\MLMC)}+\frac13\Yhat_{L,\Np_L}$.
To observe the improved accuracy, the statistical error must also
be reduced to match the  bias error. An increase in accuracy from second- to fourth-order accuracy is
achieved because the integrator is symmetric.

In the experiments, we apply extrapolation in the following scenarios:

\begin{description}
\item[EMGe and EMGe+] EMG/EMG+  with 
  extrapolation, increasing the weak order of
  convergence from one to two. 
\item[SEGe] SEG  with extrapolation, again increasing the weak order of
  convergence from one to two. 
\item[SVGe] SVG with extrapolation, increasing the weak order of
  convergence from two to four. Due to the fourth-order
  convergence, it is sufficient to take large time steps, and
  we choose $L=2$ and vary $\Ncoarse$ rather
  than $L$.
\end{description}

\subsection{Langevin equation for the damped harmonic oscillator}
\label{eg:ho}
\begin{figure}
 \begin{center}
  \includegraphics[width=0.45\linewidth]{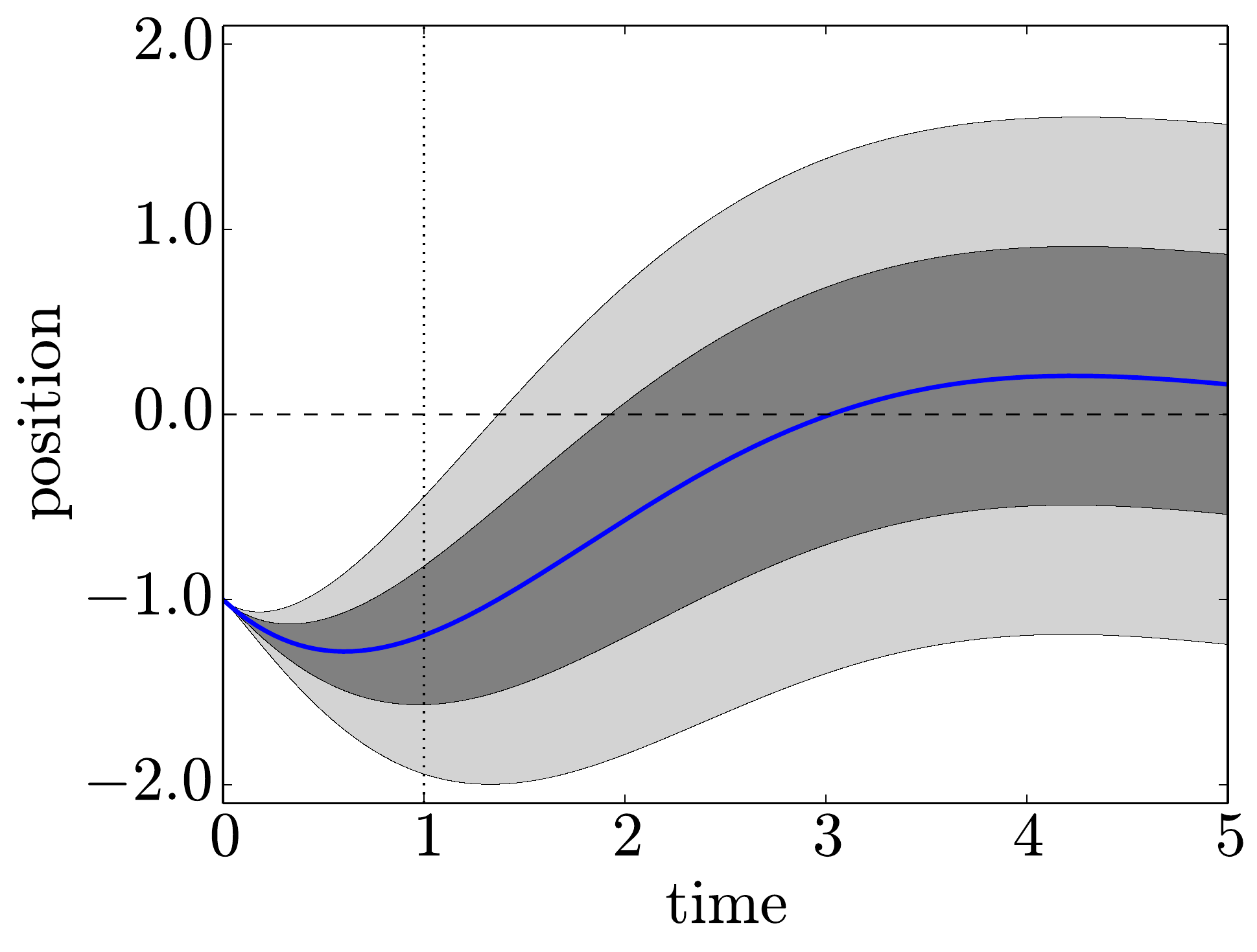} \hfill
  \includegraphics[width=0.45\linewidth]{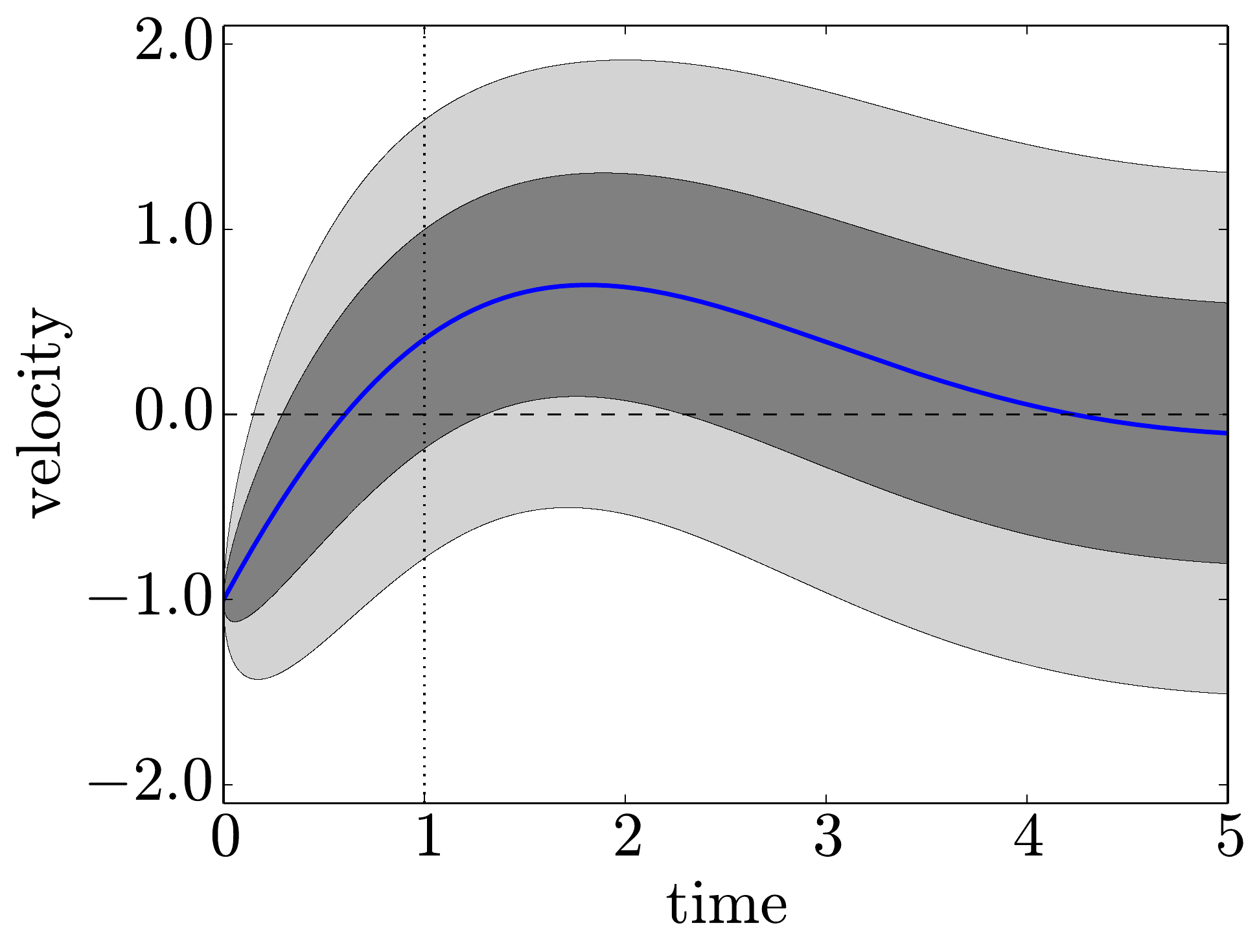}\\
  \caption{The position (left) and velocity (right) of a randomly
    forced harmonic oscillator with damping for
    $\omega_0=1=\lambda=\sigma$. The mean value is shown together
    with one (dark gray) and two standard deviations (light
    gray).}
  \label{fig:ExactSolution}
 \end{center}
\end{figure}
We first consider \cref{eq:langevin} with $d=1$ and
\begin{equation}
  V(Q) %
  = \frac{1}{2}\omega_0^2 Q^2.
\end{equation}
Physically, with this potential, \cref{eq:langevin}
describes a randomly forced harmonic oscillator with
resonance frequency $\omega_0$ and damping parameter
$\lambda$; the strength of the Gaussian forcing is given by
$\sigma$. For $\omega_0=0$ (i.e., in the absence of a
potential), the SDE can be interpreted as a model for the
dispersion of an atmospheric pollutant in a one-dimensional
turbulent velocity field (see \cite{Rodean1996}). In this
case, $\sigma^2/(2\lambda)$ is the turbulent-velocity
variance and $1/\lambda$ the velocity
relaxation-time.  
In \cref{fig:ExactSolution}, the marginal distributions for the
position and velocity are visualised as a function of
$t$ for the first set of parameters used in the numerical
experiments ($\omega_0=1=\lambda=\sigma$ and $P(0)=Q(0)=-1$).

We choose this simple example, for which we know the analytical
solution, to verify the correctness of our code and to quantify
numerical errors; exact solutions of the Langevin equation are also described in \citep{Risken1996}. As the system is linear, the joint pdf of $Q$ and $P$ is
Gaussian and is defined by their mean and
covariance. Denoting $\vec{X}(t)=(Q(t),P(t))^\trans$ and the
initial solution by $\vec{X}_0 = \vec{X}(t=0) = (Q(t=0),P(t=0))^\trans$,
we have
\begin{equation}
  \vec{X}(t)%
  = \exp\bp{-\Lambda t}\vec{X}_0%
  + \int_0^t\exp\bp{-\Lambda(t-s)}\vec{\Sigma}\,dW(s)\label{eq:ExactSolutionHarmonic}
\end{equation}
with
\[
  \Lambda \coloneq \begin{pmatrix}0 & -1\\\omega_0^2 &
    \lambda \end{pmatrix},%
  \qquad \vec{\Sigma} 
  \coloneq \begin{pmatrix}0\\\sigma\end{pmatrix}.
\]
$\vec{X}(t)$ follows a Gaussian distribution with mean
\begin{equation}
  \mean{\vec{X}(t)}%
  =\exp\bp{-\Lambda
      t}\vec{X}_0\label{eq:bc}
\end{equation} and covariance matrix
\begin{equation}
  B(t)
    \coloneq \int_0^t 
        \exp\bp{-\Lambda(t-s)}
          \vec{\Sigma}\vec{\Sigma}^\trans
        \exp\bp{-\Lambda^\trans(t-s)} \, ds\,,\label{eq:bd}
\end{equation}
which can easily be evaluated using
a computer algebra system.

\subsubsection{Numerical results}

We compute $\mean{\phi(\vec X(1))}$ for
{$\phi(Q,P)=\exp(-2(P-0.5)^2)\sqrt{2/\pi}$} and the following
set of parameters:
\begin{enumerate}
\item $\omega_0=1$, $\lambda= 4$ and $\sigma=2$.
\item $\omega_0=1$, $\lambda= 9$ and $\sigma=3$.
\end{enumerate}
The initial position and velocity were set to
$Q(t=0)=P(t=0)=-1$ in both cases. Errors are computed using
the exact value computed from \cref{eq:bc,eq:bd}.
{The exact values are 
  $0.447904416997582$ and $0.418086875513087$, respectively.}  The CPU time scaled
by $\epsilon^{-2}$ and the error {(bias error plus one standard
deviation)} scaled by $\epsilon$ are
plotted in \cref{fig:Results2_1_1,fig:Results2_2_1} against
$\epsilon$. The scaling means we expect both graphs to be
flat. {We observe for both parameter sets that the
  integrators based on the exact OU process are
  the most efficient for small $\epsilon$.  Even though
  SVGe uses a weak fourth-order accurate integrator, the
  complexity of MLMC cannot be reduced beyond
  $\order{\epsilon^{-2}}$ and it is the same as for the other
  integrators. The improvements come by improving constants,
  in this case by about a factor 4 in comparison to EMG. For
  the second set of parameter values in
  \cref{fig:Results2_2_1}, the relaxation time is shorter
  and the noise is larger, and the improvement due to the splitting
  methods is even more pronounced (factor 10).}
\begin{figure}[t!]
  \centering
   \includegraphics{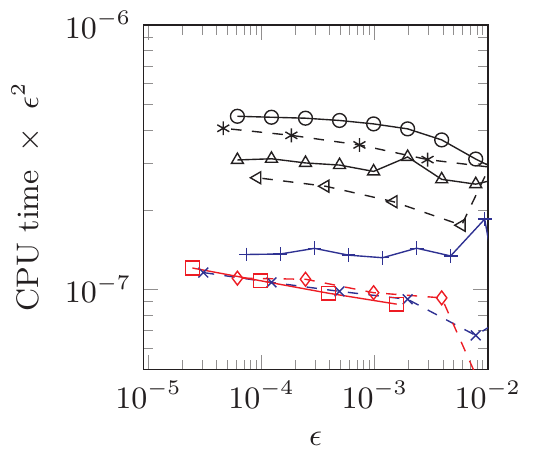}
   \includegraphics{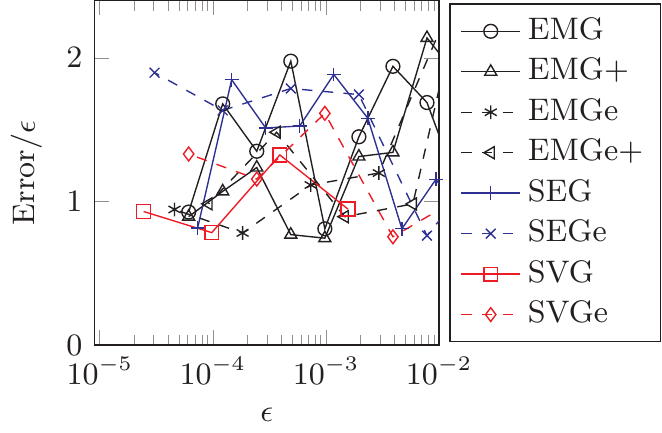}\vspace{-2ex}
  \caption{For the harmonic oscillator with parameter set
    1. The left-hand plot shows the CPU time for a given
    value of $\epsilon$; the time is scaled by $\epsilon^{-2}$
    and this leads to a nearly flat profile in each case. The
    right-hand plot shows the bias error plus one standard deviation; 
    the errors are divided by $\epsilon$ to show
    both mean and standard deviation are $\order{\epsilon}$.} 
  \label{fig:Results2_1_1}\vspace{4.5ex}
  \centering  
 \includegraphics{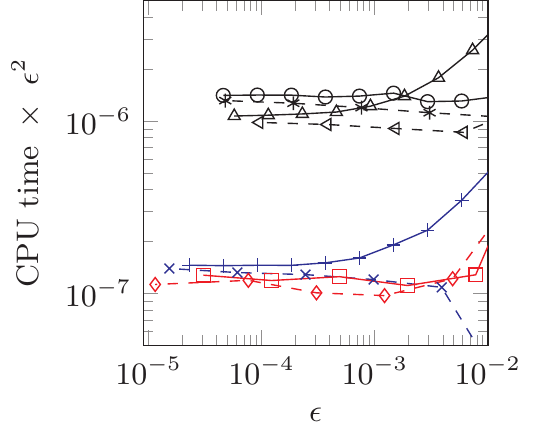}
 \includegraphics{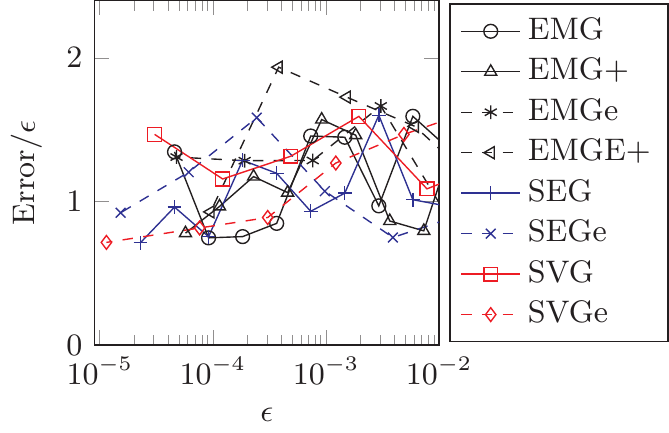}\vspace{-2ex}
  \caption{For the harmonic oscillator
    with parameter set 2. Compared to
    \cref{fig:Results2_1_1}, the difference between the
    splitting methods and Euler--Maruyama is significantly
    larger.}   \label{fig:Results2_2_1}

  \label{fig:Results2_2_1}
\end{figure}  

{In order to take large time-steps, 
 it is necessary to ensure
  the stability of the integrator.  It is well known from
  deterministic differential equations that most explicit
  integrators will have a stability constraint on the
  time-step size. This is the same for SDEs and such
stability constraints may severely restrict the number of
levels that can be employed in the MLMC method and thus its
efficiency \citep{arnulf:2012,AbdulleBlumenthal2013}. Exact sampling of the
  Ornstein--Uhlenbeck process poses no stability
  constraints, allowing for smaller values of $\Ncoarse$ 
  and thus for larger numbers of levels in MLMC in the case of
  splitting methods. For example, in the above simulations, 
  increasing the number of time steps from $\Ncoarse=4$ to
  $\Ncoarse=8$ in Euler-Maruyama (cf.~EMG and EMG+, 
  as well as EMGe and EMGe+) lead to an improvement in efficiency.
  The same change has no effect in SEG. However, the symplectic
  methods we are using for the Hamiltonian part are explicit and 
  have their own stability constraint \citep{skeel02:_langev},
  somewhat limiting this benefit of splitting methods.}

\subsection{Double-well potential}\label{subsec:num2a}\label{eg:dw}

We now change the potential and consider the double-well potential
\[
V(Q)=\frac{\omega_0^2}{8Q_{\min}^2}(Q^2-Q_{\min}^2)^2,
\]
where $Q_{\min}$ and $\omega_0$ are parameters.
{We compute $\mean{\phi(\vec X(T))}$ for
  $\phi(Q,P)\coloneq(Q+Q_{\min})^2+P^2$ (note
  $(Q+Q_{\min})^2$ takes distinct values at the bottom of
  the wells $Q=\pm Q_{\min}$). %

 For the numerical experiments in \cref{fig:Results3_1_1}, we choose parameter values
 $Q_{\min}=\omega_0=1$, $\lambda=2$,  $\sigma=4$, and
 take initial data $Q(t=0)=P(t=0)=-1$.  The scaled CPU time
 and error for $T=1$ are plotted against $\epsilon$ in
 \cref{fig:Results3_1_1}, where errors are computed relative
 to a numerically computed value given by $
 4.52782626985$. It is noticeable again that the splitting
 methods and especially the symplectic
 Euler-based methods are most efficient.

 In \cref{fig:Results3_2_1}, we explore the behaviour of the
 algorithm as we increase the length of the time interval
 $T$. For the plot, we scale the CPU time by $\epsilon^{-2}T$; the
 computation time scales linearly with the number of time
 steps and, by scaling by $T$, we see how the MLMC algorithm
 behaves with increasing $T$. The errors are computed
 relative to the numerically computed values $6.11075602345$
 for $T=2$; $7.11570774835$ for $T=4$; and $7.2125872733$
 for $T=8$. The values for $T=4$ and $T=8$ are close, which
 indicates the system has moved close to the invariant
 measure by this time. In each case, SEG is most efficient
 and we see the measure of CPU time $\times \epsilon^2/T$
 decrease from about $5\times 10^{-5}$ for $T=1$ to about $
 10^{-4}$ for $T=8$. The profiles are also less flat as $T$
 is increased, indicating that the time steps may not be
 small enough to have entered the asymptotic regime. {It is natural
 that the gains become less pronounced, when we come close to the
 invariant measure and the coupling between levels has decayed.}} 
\begin{figure}[t]
   \centering
   \includegraphics{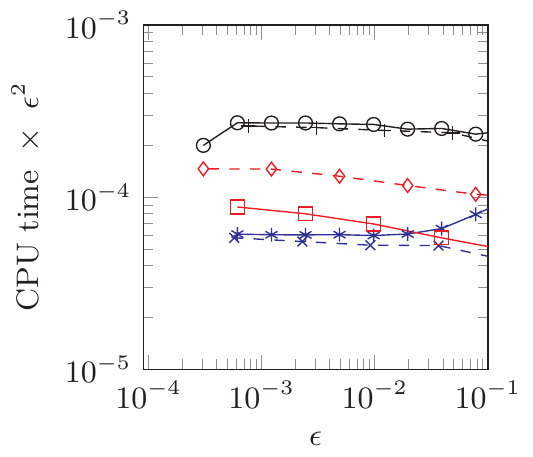} 
   \includegraphics{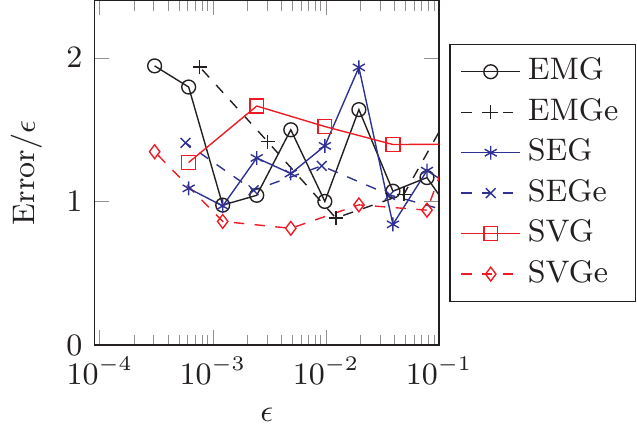} \vspace{-2ex}
   \caption{Numerical results for the double-well
     potential (plots as above).}  \label{fig:Results3_1_1}\vspace{4ex}
    \centering
    \includegraphics{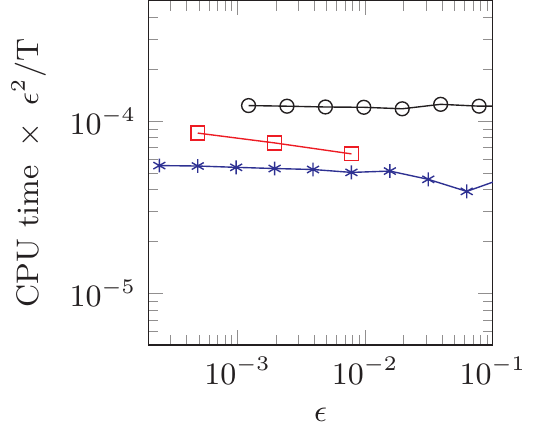}
    \includegraphics{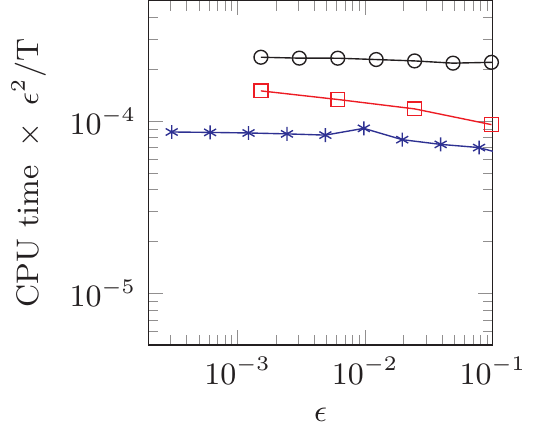} \\
    \includegraphics{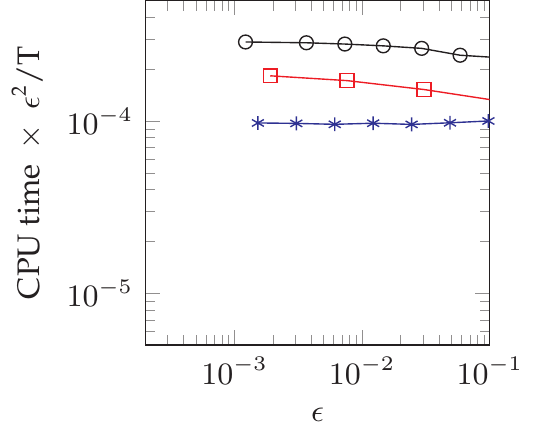}
    \includegraphics{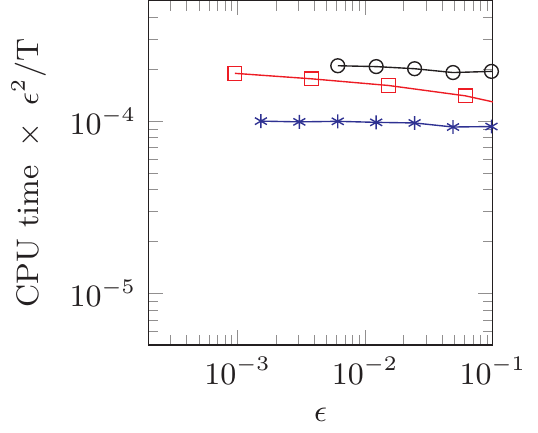} \vspace{-2ex} %
    \caption{From top-left to bottom-right, plots of CPU
      time scaled by $T\epsilon^{-2}$ for $T=1,2,4,8$ and
      the double-well potential with $\omega_0=1$,
      $\lambda=2$, $\sigma=4$.} \label{fig:Results3_2_1}
 \end{figure}

\section{Further enhancements}
{
\cref{maint,maint1} provide a route to analysing the MLMC
entirely by weak-approximation properties of the numerical
method. This has a number of advantages: from the
theoretical point of view, the analysis works for a wider
set of test functions compared to the analysis of
\citep{MR2436856}, which demands that quantities of interest
$\phi$  are globally Lipschitz continuous. From the
algorithmic point of view, the order of weak convergence is
determined by moment conditions up to a given degree
depending on the order of convergence. There are a number of
ways to satisfy these conditions. It is widely known
\citep{MR1214374} that the Gaussian random variables can be
replaced by discrete random variables without disturbing the
weak order of convergence. The obvious question then is whether we can
use discrete random variables to our advantage also in the
context of MLMC.

MLMC depends crucially on the fact that the sum of two
independent Gaussian random variables is also Gaussian. This
allows increments to be generated on the fine levels and
combined to give a random variable with the same
distribution on the next coarser level, using \cref{ut}. Discrete
variables do not have this property. While
\cref{maint} implies the coupling condition of
\cref{theo:ComplexityTheorem}(iii), the sum of two
three-point random variables is not a three-point random variable
and the telescoping sum breaks down. In general,
using  discrete random variables with MLMC introduces extra
error due to the telescoping sum no longer being
exact. Though   \citep{belomestny14:_multil} provides an
approach  that preserves the
telescoping sum by using a different discrete random
variable on each level. 
Here we do not follow this route. Instead, we use the same
discrete random variable on each level, accepting the
additional bias error that this introduces, which crucially
is of higher order. To control this additional bias and to
ensure the total error is still below our chosen tolerance,
we change the number of levels $L$ and the coarsest mesh
size $h_0$.  Discrete random variables are cheaper to
generate than Gaussian random variables and the coarsest
level can be evaluated exactly, which we exploit to achieve
a significant speed-up in the small noise case.}

\subsection{Random variables with discrete distribution}
\label{seq:DiscreteDistribution}

{The modified equations are unchanged if the Gaussian random
  variables in the integrator are replaced by random
  variables with the same moments to order five (including
  all cross moments to order five arising from the
  doubled-up system).} For example, we can replace samples
of \iid $\Nrm(0,1)$ random variables by \iid samples of the
random variable $\zeta$ with distribution
 \begin{equation}\label{16}
\prob{\zeta=0}%
=\frac 23,\qquad
\prob{\zeta=\pm \sqrt 3}%
=\frac 16;%
\end{equation}
or 
\begin{equation}\label{17}
\prob{\zeta=\pm \sqrt{3+\sqrt 6}}%
=c,\qquad%
\prob{\zeta=\pm \sqrt{3-\sqrt 6}}%
=\frac 12-c,\qquad c\coloneq \frac12 \pp{1-\frac{3+\sqrt 6}6}.%
\end{equation}
We refer to $\zeta$ as the three- and four-point
approximations to the Gaussian, respectively.  This is a
well-known trick for weak approximation of SDEs, e.g.
\citep[\S14.2]{MR1214374}.  The  approximations
have a number of advantages, as $\zeta$ is quicker to sample
than a Gaussian and, due to the finite number of states,
averages of functionals of $\zeta$ can be computed exactly.

\subsection{Exact evaluation of the  coarse-level expectation}%
\label{sec:ExactCoarseLevel}%
For all our integrators, the evaluation of the coarse-level estimator
$\Phat_{0}$ with time step $h_{0}=T/\Nt_{0}$ is
the computationally most expensive part of the MLMC
algorithm: even though the number of time steps and hence
the number of samples per path is small, a  large number
of individual paths needs to be evaluated to reduce the
variance of the coarse-level estimator.
This cost can be reduced dramatically if a discrete distribution
as discussed in \cref{seq:DiscreteDistribution} is used for the
individual samples $\vec{\xi}_n$. In this case, a significantly
cheaper estimator, which does not rely on Monte Carlo
sampling, can be constructed. If the random numbers $\vec
\xi_1,\dots,\vec \xi_{\Nt_{0}}$ for each path are drawn from
the three-point approximation in \cref{16}, there is only a
finite number $n_{\vec\xi}$ of possible samples $\vec
\xi^{(i)}=\{\vec\xi^{(i)}_1,\dots,\vec\xi^{(i)}_{\Nt_{0}}\}$,
each with associated probability $\mathbb{P}(\vec{\xi}^{(i)}) =
\mathbb{P}(\vec\xi_1=\vec\xi_1^{(i)})\cdots
\mathbb{P}(\vec\xi_{\Nt_{0} }=\vec\xi_{\Nt_{0}}^{(i)})$.
The expectation value of the quantity of interest can be
calculated exactly on the coarsest level as
\begin{equation}
  \Yhat_{0}^{\operatorname{exact}}= \Phat^{\operatorname{exact}}_{0}%
  = \sum_{i=1}^{n_{\vec{\xi}}} \mathbb{P}(\vec{\xi}=\vec{\vec\xi}^{(i)})\QOI_{0}^{(i)}.
  \label{eq:Yexact}
\end{equation}
{For the three-point approximation, for example, we need to choose
from the $3^d$ possible values of $\vec{\xi}_n$ in each of the $\Nt_{0}$ time steps,} so
$n_{\vec\xi} = (3^d)^{\Nt_{0}}$ is the number of
different samples of $\vec{\xi}$.  Since the estimator
contains no sampling error, its variance is zero. In
Algorithm \ref{alg:MLMC}, we can replace
$\Yhat_{0,\Np_{0}} \mapsto
\Yhat_{0}^{\operatorname{exact}}$ and
$\Vhat_{0,\Np_{0}}\mapsto 0$ in lines
\ref{eq:algEstimators} and
\ref{eqn:Nestimator}. Effectively, this implies that the sum
in line \ref{eqn:Nestimator} only runs from $j=1$ to $L$
and it is not necessary to evaluate $\Np_{0}^+$.

Naively, the computational complexity of evaluating
\cref{eq:Yexact} is given by the
product of the number of different samples and the number of
time steps, $n_{\vec{\xi}}\times
\Nt_{0}=\Nt_{0}(3^{d\Nt_{0}})$. However,
using a recursive algorithm,
the computational complexity can be
reduced to the number of nodes in the product-probability
tree, which is only $\order{n_{\vec{\xi}}} =
\mathcal{O}(3^{d\Nt_{0}})$ Nevertheless, 
this still grows exponentially with {the number $\Nt_{0}$ of coarse time steps}
and so \cref{eq:Yexact} is only competitive
for small values of $\Nt_{0}$ and $d$. However, exact
evaluation can reduce the overall cost of the algorithm dramatically
and this is exploited to significant advantage in
\cref{subsec:num_discrete}.

We now state and prove a modified complexity theorem that allows for additional bias  to be introduced between levels, as well as for a different computational cost on the coarsest level.

Let $\Ptilde_\ell$ be the estimator corresponding to $\Phat_\ell$, but with increments given by \cref{ut}. For Gaussian increments these estimators are the same, but they are different when we use 3-point or 4-point approximations. Recall that the fine, level $\ell$, sample in each of the estimators $\Yhat_{\ell,\Np_\ell}$ uses increments sampled directly from the 3-point or 4-point distribution, while the coarse, level $\ell - 1$, sample  is computed using two consecutive fine increments and formula \cref{ut}.
\begin{theorem} \label{theo:modifiedComplexity}
Let us replace Assumption (ii) of  \cref{theo:ComplexityTheorem} by
\begin{enumerate}
\item[(ii)] $\expect{\Yhat_{0}} = \expect{\Phat_\ell}$ \ and \ $\left|\expect{\Phat_\ell-\Ptilde_{\ell}}\right| \le c_0 h_{\ell}^\gamma$,
\end{enumerate}
for some positive constants $c_0$ and $\gamma > \alpha \ge \frac{1}{2}$. We suppose that all the other assumptions of  \cref{theo:ComplexityTheorem} hold, except that $\Cost{MLMC}_0$ is not necessarily assumed to be bounded by $c_3 \Np_{0} h_{0}^{-1}$ any longer. Then, there exists a positive constant $c_5$ such that for any $\epsilon<1/e$, there are values $\Nt_0$, $L$ and $\Np_\ell$ for which $\Yhat_{\{\Np_\ell\}}$ from \cref{eq22} has a MSE $ < \epsilon^2$ and a computational complexity $\Cost{MLMC}$ with bound
\begin{equation}
  \Cost{MLMC} \le \Cost{MLMC}_0  + c_5 \epsilon^{-2 + 1/\gamma} \, .
\end{equation}
\end{theorem}

\begin{proof}
We only require slight modifications in the proof of \citep[Theorem 3.1]{MR2436856} to prove this result. In particular, it is sufficient to choose
\[
L = \ceil*{\frac{\log_2(\sqrt{3} c_1 T^\alpha \epsilon^{-1})}{\alpha}}
\]
to bound the bias on the finest level. The factor $\sqrt{3}$ appears, since we now have three error contributions, the bias on the finest level, the bias between levels and the sampling error, and since we require each of these contributions to the MSE to be less than $\epsilon^2/3$.

To guarantee that the bias between levels is less than $\epsilon^2/3$, note that due to assumption (ii) we have
\[
\left|\sum_{\ell=0}^{L-1} \expect{\Phat_\ell-\Ptilde_{\ell}}\right| \le \sum_{\ell=0}^{L-1} \left|\expect{\Phat_\ell-\Ptilde_{\ell}}\right| \le c_0 h_{0}^\gamma \sum_{\ell=0}^{L-1} 2^{-\ell \gamma} < \frac{\sqrt{2} c_0}{\sqrt{2}-1} \, h_0^\gamma\,,
\]
and so a sufficient condition is $h_0 \le c_6 \epsilon^{1/\gamma}$ with $c_6 = \left(\frac{\sqrt{3}}{c_0}\left(1 - \frac{1}{\sqrt{2}}\right)\right)^{1/\gamma}$ .

Finally, setting
\[
\Np_\ell = \ceil*{\frac{3 c_2 h_0^2}{\sqrt{2}-1} \epsilon^{-2} 2^{-3\ell/2}}
\]
and exploiting standard results about geometric series, we get
\[
\sum_{\ell=1}^L \variance{\Yhat_{\ell,\Np_\ell}} \le \frac{\epsilon^2}{3} \left(\sqrt{2}-1\right) \sum_{\ell=1}^L \frac{c_2 h_\ell^2}{c_2 h_0^2} 2^{3\ell/2} \le \frac{\epsilon^2}{3} \left(\sqrt{2}-1\right) \sum_{\ell=1}^L 2^{-\ell/2} \le \frac{\epsilon^2}{3}\,.
\]
The computational cost can then be bounded by
\[
\Cost{MLMC} \le \Cost{MLMC}_0 + c_3 \sum_{\ell=1}^L \Np_\ell h_\ell^{-1} \le \Cost{MLMC}_0 + \frac{3 c_2 c_3 h_0}{\sqrt{2}-1} \epsilon^{-2} \sum_{\ell=1}^L 2^{-\ell/2}
\]
which leads to the desired bound with $c_5 = 3 c_2 c_3 c_6 \left(\sqrt{2}-1\right)^{-2}$. (Note that as in \citep{MR2436856} this (optimal) choice of $\Np_\ell$
is obtained by minimising the cost on levels 1 to $L$ subject to the constraint that the sum of the variances is less than $\epsilon^2/3$.)
\end{proof}
If we use a $q$-point approximation and the expected value
on the coarsest level is computed excactly, as described in
\cref{sec:ExactCoarseLevel}, then $\Cost{MLMC}_0 =
\mathcal{O}(q^{d\Nt_0}) =
\mathcal{O}\Big(\eta^{\epsilon^{-{1}/{\gamma}}}\Big)$,
for some $\eta > 1$. Hence, the total cost grows
exponentially with $\epsilon$, as expected. However, for
practically relevant values of $\epsilon$, the exponential
term may not be dominant and we may get significant
computational savings, as we will see in the next
section. Note that $\gamma = 2$ for the three-point and
$\gamma=3$ for the four-point case, leading to a cost of
$\mathcal{O}(\epsilon^{-3/2})$ and
$\mathcal{O}(\epsilon^{-5/3})$ for the computation of the
correction terms on levels 1 to $L$, respectively.

Since the sampling of discrete random variables is
significantly cheaper, it may also be of interest to use
standard Monte Carlo on the coarsest level, as in the
earlier sections of this paper. If we slightly increase the
constant in the formula for $N_\ell$, $\ell=1,\ldots,L$, in
the proof of \cref{theo:modifiedComplexity} and choose
$\Np_0 = \mathcal{O}(\epsilon^{-2})$ such that the total
variance over all levels is below $\epsilon^2/3$, then the
dominant cost will be $\Cost{MLMC}_0 = \mathcal{O}(\Np_0
h_0^{-1})$, and so $\Cost{MLMC} \le c_5^*
\epsilon^{-2-1/\gamma}$, which will be
$\mathcal{O}(\epsilon^{-5/2})$ and
$\mathcal{O}(\epsilon^{-7/3})$ in the three- and four-point
cases, respectively. However, in practice $c_5^*$ is
significantly smaller than the constant $c_4$ in Theorem
\ref{theo:ComplexityTheorem}, so that for moderate values of
$\epsilon$, the use of discrete random variables will pay
off.
\subsection{Numerical experiments with discrete random variables}
\label{subsec:num_discrete}
\begin{figure}[t]
    \centering
    \includegraphics[width=0.5\linewidth]{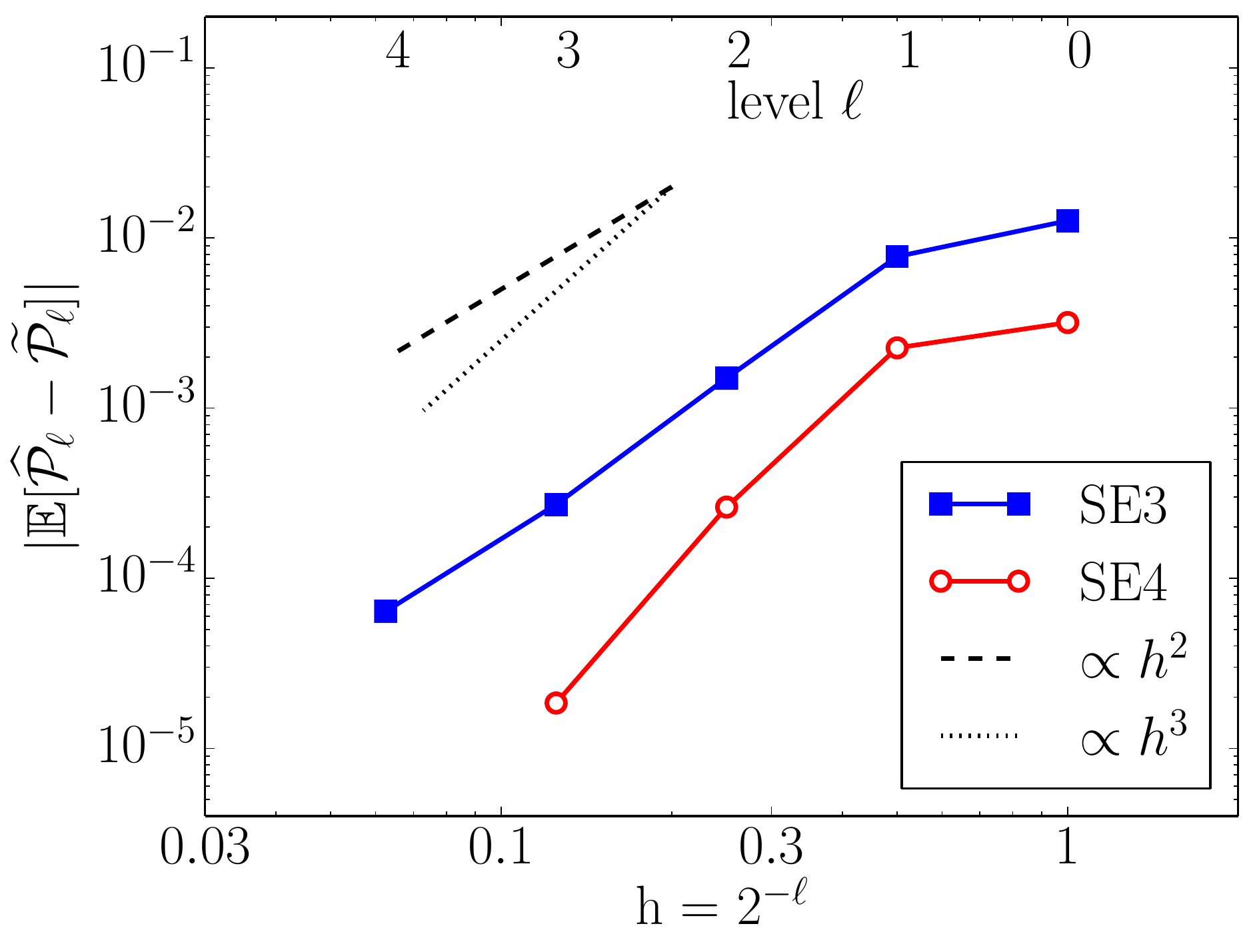}\vspace{-2ex}
    \caption{$h$-dependency of the additional bias term $|\expect{\Phat_\ell-\Ptilde_{\ell}}|$ (see (ii) in  \cref{theo:modifiedComplexity}) for $\sigma_0=\lambda=1$, $\sigma=0.4$ computed using the symplectic Euler/exact OU splitting. Results are shown both for three-point (SE3) and four-point (SE4) random variables.}
\label{fig:additional_bias}
 \end{figure}
 We carry out numerical experiments as in \cref{eg:ho} with the
 damped harmonic oscillator, but change the parameters
 slightly to $\omega_0=\lambda=1$, $\sigma=0.4$ (i.e.,
 smaller noise). Instead of sampling from a Gaussian
 distribution, we use discrete random numbers, which
 introduce an additional bias as discussed above. To
 quantify this bias numerically, we plot the difference
 $\left|\expect{\Phat_\ell-\Ptilde_{\ell}}\right|$ in
 \cref{fig:additional_bias} for the symplectic Euler/exact
 OU method both for three-point (SE3) and four point (SE4)
 distributions. {The figure shows that, as predicted in
   \cite{MR1352201}, the additional bias is proportional to
   $h^2$ for SE3 and to $h^3$ for SE4. We have also studied
   the dependence on the noise term (not shown here) and
   found that, as $\sigma$ gets smaller, the additional bias
   is reduced very rapidly (proportional to $\sigma^3$ and
   $\sigma^4$, respectively).
\begin{figure}[t]
    \centering
    \includegraphics{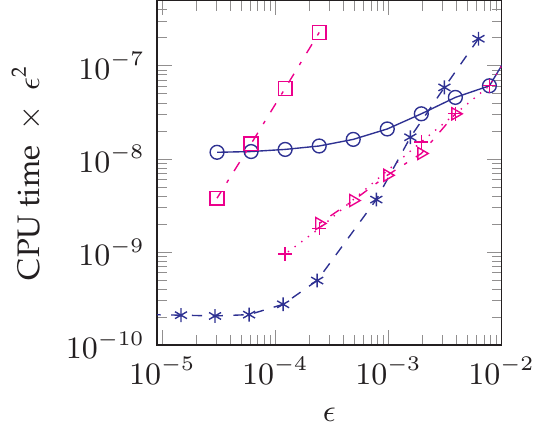}
    \includegraphics{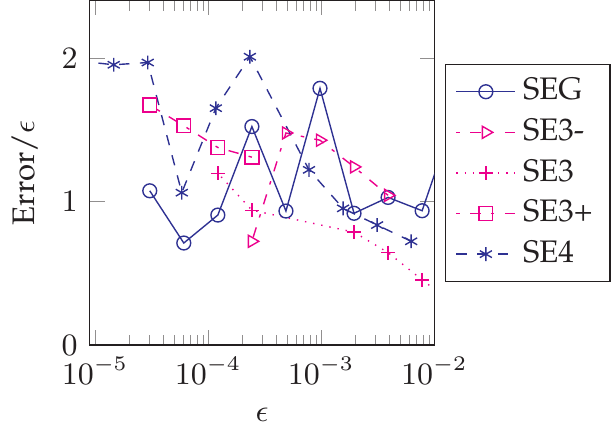}\vspace{-2ex}
    \caption{Harmonic oscillator with $\lambda=1$ and
      $\sigma=0.4$ computed using the symplectic Euler/exact
      OU splitting method using three-point (SE3-, SE3, SE3+) and
      four-point (SE4) random variables.} \label{fig:discrete_distribution}
 \end{figure}

 For the same setup, we measure the computational cost and
 the total error (consisting of the statistical error,
 discretisation error and the additional bias introduced by
 sampling from discrete distributions). We calculate the
 same quantity of interest as in \cref{eg:ho}.
 \cref{fig:discrete_distribution} shows the results both for
 Gaussian random variables (SEG) and for the three- and
 four-point distributions (SE3 and SE4) with $\Nt_0 =
   8$. For the discrete distributions, the coarse-grid
   expectation value is calculated exactly.  For the
   three-point distribution, we also varied the number of
   time steps on the coarsest level and use $\Nt_0=4$
   (SE3-), $\Nt_0=8$ (SE3) and $\Nt_0=16$ (SE3+). In each
 case, we only show results up to the point where the
 additional bias error becomes too large. For fixed
 $\epsilon$, the SE3+ method is more expensive than SE3 and
 SE3-, since the cost of the exact coarse-level evaluation
 grows exponentially with the number of time steps. On the
 other hand, using smaller time steps on the coarsest level
 allows the use of this method for smaller values $\epsilon$
 where the additional bias becomes too large for SE3- and
 SE3. The additional bias in the SE4 method is so small that
 the method can be used up to values as small as
 $\epsilon=10^{-5}$. Comparing the cost of this method to
 the Gaussian case shows that using a discrete four-point
 distribution is more than 50-times faster in this case.

We conclude that, if used with caution, approximating the Gaussian
increments in \cref{eq:1} by discrete approximations and calculating 
the coarse-level expectation value exactly can significantly improve
the efficiency of the multilevel method.}

\section{Conclusion}

{\cref{tab1} summarises our findings:} MLMC gives a
significant speed-up over the traditional Monte Carlo
computation of averages and, even though the optimal
complexity estimate $\order{\epsilon^{-2}}$ {for 
Monte Carlo-type methods} holds
for all the integrators under study, there is significant
variation between the integrators. Splitting methods are
particularly appropriate for the Langevin equation and
using the exact OU solution yields a  more stable
integrator than Euler--Maruyama, even though both integrators are
explicit.  In the experiments, the difference in
computation time between Euler--Maruyama and the splitting
methods is greater {when the dissipation $\lambda$ 
is higher, since Euler--Maruyama suffers from a more severe 
time-step restriction (cf.~Figures \ref{fig:Results2_1_1}--\ref{fig:Results3_2_1}).}

\begin{table}[t]
  \centering
  \begin{tabular}{c||c c||c c|}
& \multicolumn{2}{c||}{Harmonic Oscillator (Set 2)}&
                                                   \multicolumn{2}{c|}{Double-well
                                                    Potential} \\
    & $\epsilon=2.4\times 10^{-4}$  &ratio& $\epsilon=2.44\times 10^{-3}$
    &ratio\\
\hline
    MC w.~EMG& 467  sec& 13$\times$ slower&
    1710 sec & 378$\times$ slower\\
    MLMC w.~EMG& 33.8 sec&1& 45.2 sec& 1\\
    MLMC w.~SEGe& 2.15 sec& 15$\times$ faster& 10.5 sec& 4.3$\times$ faster\\
  \end{tabular}
\caption{{Comparison of Monte Carlo with Euler--Maruyama, MLMC with
  Euler--Maruyama and MLMC with the
  symplectic Euler/OU integrator and extrapolation (using Gaussian
  increments).}}\label{tab1}\vspace{3ex}
\end{table}

This paper also introduced an alternative analysis method
for MLMC based on modified equations. It provides a
convenient approach to MLMC through weak-approximation
theory; strong-approximation theory is only needed to relate
the original and modified equations and not the numerical
methods. This  accommodated the use of the
splitting method and the exact OU
solution easily. 

{ The weak-approximation analysis motivated the use of
  discrete random variables, such as three- and four-point
  approximations to the Gaussian. In an example with small
  noise ($\sigma=0.4$ and $\lambda=1$), we saw between one and two orders of
  magnitude speed-up for a useful range of $\epsilon$
  because we can evaluate the coarse level exactly. This method
  is easy to implement and it works well because the dominant cost
  lies on the coarsest level for these
  problems. While the speed improvements are impressive,
  this method should be used with care as it introduces an
  extra bias error. The extra bias can be estimated
  as shown in \cref{fig:additional_bias}. The improvement would
  be less dramatic in higher dimensions as the number of
  samples required would increase dramatically and it may be
  impossible to compute the coarse level exactly. As an interesting
  side result, we proved a modified complexity theorem that allows for
  extra bias to be introduced between levels in MLMC.}

\appendix{}
\section{Proof of \cref{lemma:1}}\label{useful_lem}

\begin{proof}  This is a standard Gronwall argument with the 
Ito isometry. The integral equation for the difference $\vec
  X(t)-\vec X_h(t)$ is
    \begin{align*}
      \vec X(t)-\vec X_h(t)&=\int_0^t\pp{ \vec f(\vec
      X(s))-\tvecf(\vec X_h(s))}\,ds\\
      &\quad+\int_0^t \pp{G(\vec X(s))-
      \tG(\vec X_h(s))}\,d\vec W(s),
    \end{align*}
    where $\tvecf\coloneq\vec f+ h \vec f_1$ and $\tG\coloneq G+h G_1$. By conditions (i) and (ii),
 \begin{align*}%
   \norm{\vec f(\vec X(s))-\vec f(\vec X_h(s)) -h \vec
     f_1(\vec X_h(s))}_{L^2(\Omega,\real^d)}%
   \le& L \norm{\vec X(s)-\vec
     X_h(s)}_{L^2(\Omega,\real^d)}%
   + C_1 h
  \end{align*}
  and similarly for $G$.  Assume $\vec X(0)=\vec
  X_h(0)$. 
The Ito isometry and Jensen's inequality give
\begin{align*}
  \mean{ \norm{ \vec X(t)-\vec X_h(t)}^2}%
  \le&2t\int_0^t \mean{\norm{\vec f(\vec X(s))-\tvecf(\vec
      X_h(s))}^2}\,ds\\%
  &+2\int_0^t \norm{G(\vec X(s))- \tG(\vec
    X_h(s))}^2_\rmF\,ds\\
  \le&4t\int_0^t L^2\mean{\norm{\vec X(s)-\vec X_h(s)}^2}+C_1^2
    h^2\,ds\\%
  &+4\int_0^t L^2\mean{\norm{\vec X(s)-\vec X_h(s)}^2}+C_1^2
    h^2\,ds.
    \end{align*}
    Gronwall's inequality gives $\order{h^2}$ bounds on
    $\mean{\norm{\vec X(t)-\vec X_h(t)}^2}$. A similar argument
    can be applied to $\vec Y$ and applying the Lipschitz
    condition on $\psi$ completes the proof.
\end{proof}%
%

\end{document}